\documentclass[preprint]{elsarticle}
\usepackage[latin9]{inputenc}
\usepackage{float}
\usepackage{amsmath}
\usepackage{amssymb}
\usepackage{stmaryrd}
\usepackage{graphicx}
\usepackage[unicode=true,
 bookmarks=false,
 breaklinks=false,pdfborder={0 0 1},backref=section,colorlinks=false]
 {hyperref}

\makeatletter

\floatstyle{ruled}
\newfloat{algorithm}{tbp}{loa}
\providecommand{\algorithmname}{Algorithm}
\floatname{algorithm}{\protect\algorithmname}

\usepackage{lineno}
\modulolinenumbers[5]

\journal{Linear Algebra and its Applications}

\usepackage[normalem]{ulem}
\usepackage{dsfont}
\usepackage{todonotes}
\floatstyle{ruled}
\newfloat{algorithm}{tbp}{loa}
\providecommand{\algorithmname}{Algorithm}
\floatname{algorithm}{\protect\algorithmname}

\usepackage{amsthm}

\newtheorem{lem}{Lemma}\newtheorem{theorem}{Theorem}\newtheorem{Definition}{Definition}\newtheorem{prob}{Problem}\newtheorem{prop}{Proposition}\newtheorem{coro}{Corollary}

\@ifundefined{showcaptionsetup}{}{%
 \PassOptionsToPackage{caption=false}{subfig}}
\usepackage{subfig}
\makeatother

\begin{document}
\begin{frontmatter}

\title{\textbf{Edge-Matching Graph Contractions and their Interlacing Properties}}



\author[mymainaddress]{Noam Leiter}

\ead{noaml@campus.technion.ac.il}

\author[mymainaddress]{Daniel Zelazo}

\ead{dzelazo@technion.ac.il}

\address[mymainaddress]{Faculty of Aerospace Engineering, Technion, Haifa 32000 Israel }

\global\long\def\ones{\mathds{1}}%
\global\long\def\G{\mathcal{G}}%
\global\long\def\E{\mathcal{E}}%
\global\long\def\N{\mathcal{N}}%
\global\long\def\V{\mathcal{V}}%
\global\long\def\M{\mathcal{M}}%
\global\long\def\Y{\mathcal{Y}}%
\global\long\def\U{\mathcal{U}}%
\global\long\def\C{\mathcal{C}}%
\global\long\def\T{\mathcal{T}}%
\global\long\def\graph{\mathcal{G}=\left(\mathcal{V},\mathcal{E}\right)}%
\global\long\def\TM{T_{\left(\mathcal{T},\mathcal{C}\right)}}%
\global\long\def\HT{\mathcal{H}_{2}}%
\global\long\def\R{\mathbb{R}}%
\global\long\def\diag{\text{diag}}%
\global\long\def\trace{\text{Tr}}%
\global\long\def\Hinf{\mathcal{H}_{\infty}}%
\global\long\def\Tree{\mathbb{T}\left(\G\right)}%
\global\long\def\gcont{\G\sslash\pi}%

\begin{abstract}
For a given graph $\G$ of order $n$ with $m$ edges, and a real
symmetric matrix associated to the graph, $M\left(\G\right)\in\R^{n\times n}$,
the interlacing graph reduction problem is to find a graph $\G_{r}$
of order $r<n$ such that the eigenvalues of $M\left(\G_{r}\right)$
interlace the eigenvalues of $M\left(\G\right)$. Graph contractions
over partitions of the vertices are widely used as a combinatorial
graph reduction tool. In this study, we derive a graph reduction interlacing theorem based on subspace mappings and the minmax theory. We then define a class of edge-matching
graph contractions and show how two types of edge-matching contractions
provide Laplacian and normalized Laplacian interlacing. An $\mathcal{O}\left(mn\right)$
algorithm is provided for finding a normalized Laplacian interlacing
contraction and an $\mathcal{O}\left(n^{2}+nm\right)$ algorithm is
provided for finding a Laplacian interlacing contraction. 
\end{abstract}
\end{frontmatter}

\section{Introduction}

The effect of combinatorial operations on graph spectra is an evolving
branch of graph theory, linking together combinatorial graph theory
with the spectral analysis of the algebraic structures of graphs.
In general, there is an interest to understand how certain graph reduction
operations relate to spectral and combinatorial properties. Of particular
interest are reductions that satisfy an \emph{interlacing} property
between algebraic graph representations. Interlacing properties of
algebraic structures of graphs have been shown to have combinatorial
interpretations. Haemers used the adjacency and Laplacian matrix interlacing
to provide combinatorical results on the chromatic number and spectral
bounds \cite{haemers1995interlacing}. The neighborhood reassignment
operation has been shown to provide an interlacing of the normalized
Laplacian \cite{wu2009interlacing}, and Chen et al. provide an interlacing
result on contracted normalized Laplacians \cite{chen2004interlacing}.

Partitioning the vertices of a graph is a combinatorial operation
extensively studied in graph theory in the context of graph clustering
\cite{schaeffer2007graph} and network communities \cite{newman2004finding},
and for spectral clustering methods \cite{ng2002spectral}. Partitioning
combined with node and edge contractions along those partitions lead
to reduced order graphs. In this direction, we define edge-matching
contractions as a class of graph contractions with a one-to-one correspondence
of a subset of edges in the full order graph to those in the contracted
graph. We then explore two types of edge-matching contractions, \emph{cycle
invariant contractions} and \emph{node-removal equivalent contractions}.
Cycle-invariant contractions preserve the cycle structure of the graph
in the contracted graph, and node-removal equivalent contractions
are cases where a contraction can be obtained also from a node-removal
operation. We show how contraction of these types lead to interlacing
of the normalized-Laplacian and Laplacian graph matrices. Two algorithms
of complexity $\mathcal{O}\left(mn\right)$ and $\mathcal{O}\left(n^{2}+nm\right)$
are then provided for finding, if they exist, a cycle-invariant contraction
and a node-removal equivalent contraction respectively for a given
graph with $n$ vertices and $m$ edges.

The remaining sections of this paper are as follows. In Section \ref{sec:Interlacing-Graphs},
the interlacing graph reduction problem is presented and an interlacing graph reduction theorem is derived. In Section \ref{sec:Graph-Contractions}
we formulate the graph contraction operation for simple undirected
graphs, and introduce the class of edge-matching graph contractions
and two sub-classes of cycle-invariant and node-removal equivalent
graph contractions. In Section \ref{sec:Interlacing-Graph-Contractions},
the interlacing graph reduction problem is solved for these two classes
for the Laplacian and normalized-Laplacian matrices, and Section \ref{sec:Case-Studies}
provides case studies of the interlacing methods.

\paragraph*{Preliminaries}

The integer set $\left\{ 1,\ldots,n\right\} $ is denoted as $\left[1,n\right]$.
An undirected graph $\mathcal{G}=\left(\mathcal{V},\E\right)$ consists
of a vertex set $\V\left(\G\right)$, and an edge set $\E\left(\G\right)=\{\epsilon_{1},\ldots,\epsilon_{|\E|}\}$
with $\epsilon_{k}\in\V{}^{2}$. The order of the graph is the number
of vertices $|\mathcal{V}\left(\G\right)|$. Two nodes $u,v\in\V\left(\G\right)$
are \textit{adjacent} if they are the endpoints of an edge, and we
denote this by $u\sim v$. The neighborhood $\N_{v}\left(\G\right)$
is the set of all nodes adjacent to $v$ in $\G$. The degree of a
node $v$, denoted $d_{v}\left(\G\right)$, is the number of nodes
adjacent to it, $d_{v}\left(\G\right)=\left|\N_{v}\left(\G\right)\right|$.
A \textit{path} in a graph is a sequence of distinct adjacent nodes.
A \textit{simple cycle} is a path with an additional edge such that
the first and last vertices are repeated. A graph $\G$ is \textit{connected}
if we can find a path between any pair of nodes. A \textit{simple
graph} does not include self-loops or duplicate edges. A \textit{multi-graph}
is a graph that may include duplicate edges. We denote $\G\backslash\V_{R}$
as the graph obtained from $\G$ by removing all nodes $v\in\V_{R}\subset\V$
from $\V\left(G\right)$ and removing all edges in $\E\left(\G\right)$
adjacent to $v$. We denote $\G\backslash\E_{R}$ as a graph obtained
from $\G$ by removing all edges $\epsilon\in\E_{R}$ from $\E\left(G\right)$.
A subgraph $\G_{S}=\left(\V_{S},\E_{S}\right)$ of a graph $\graph$,
denoted as $\G_{S}\subseteq\G$, is any graph such that $\V_{S}\subseteq\V$
and $\E_{S}\subseteq\E\cap\V_{S}^{2}$. An induced subgraph $\G\left[\V_{S}\right]$
is a subgraph $\G_{S}\subseteq\G$ such that $\E_{S}=\E_{\G}\cap\V_{S}^{2}$.
An induced subgraph $\G\left[\V_{S}\right]$ is a \textit{connected
component} of $\G$ if it is connected and no node in $\V_{S}$ is
adjacent to a node in $\V\left(\G\right)\backslash\V_{S}$. The set
$\Tree$ denotes the set of all spanning trees of a connected graph
$\G$. For $\T\in\Tree$, the \emph{co-tree} graph $\G\backslash\E\left(\T\right)$
is denoted as $\C\left(\T\right)$ \cite{Godsil2001}.

\section{Interlacing Graph Reductions\label{sec:Interlacing-Graphs}}

Graph matrices are algebraic representations of graphs, and the spectral
and algebraic properties of these matrices can provide insights about
combinatorial properties of the underlying graph, e.g., Fiedler's
seminal results on the Laplacian algebraic connectivity \cite{fiedler1975property}.
The interlacing property of matrices has been extensively studied
with classic algebraic results such as the Poincare separation theorem
\citep[p. 119]{bellman1997introduction}, and matrix combinatorial
results such as the relation of equitable partitions with tight interlacing
\cite{godsil1997compact}. Here we study what types of reduced graphs
have interlacing graph matrices.

The spectrum of a real symmetric matrix $A\in\mathbb{R}^{n\times n}$
is the set of eigenvalues $\left\{ \lambda_{k}\left(A\right)\right\} _{k=1}^{n}$
where $\lambda_{k}\left(A\right)$ is the $k$th eigenvalue of $A$
in ascending order. Let $A\in\mathbb{R}^{n\times n}$ and $B\in\mathbb{R}^{r\times r}$
be real symmetric matrices with $0<r<n$. Then the eigenvalues of
$B$ \emph{interlace} the eigenvalues of $A$, denoted $B\propto A$,
if $\lambda_{k}\left(A\right)\leq\lambda_{k}\left(B\right)\leq\lambda_{n-r+k}\left(A\right)$
for $k=1,2,\ldots,r$. 
The interlacing is \textit{tight }if $\lambda_{k}\left(A\right)=\lambda_{k}\left(B\right)$
or $\lambda_{k}\left(B\right)=\lambda_{n-r+k}\left(A\right)$ for
$k=1,2,\ldots,r$. It is straight forward to show that interlacing
is a transitive property.

\begin{prop}\label{interlacing sequence} Let $A_{1}\in\mathbb{R}^{n_{1}\times n_{1}}$,
$A_{2}\in\mathbb{R}^{n_{2}\times n_{2}}$ and $A_{3}\in\mathbb{R}^{n_{3}\times n_{3}}$
be real symmetric matrices with $0<n_{3}<n_{2}<n_{1}$. If $A_{3}\propto A_{2}$
and $A_{2}\propto A_{1}$, then $A_{3}\propto A_{1}$.\end{prop}

\begin{proof} From $A_{3}\propto A_{2}$ and $A_{2}\propto A_{1}$
we have $\lambda_{k}\left(A_{2}\right)\leq\lambda_{k}\left(A_{3}\right)\leq\lambda_{n_{2}-n_{3}+k}\left(A_{2}\right)$
for $k=1,2,\ldots,n_{3}$ and $\lambda_{l}\left(A_{1}\right)\leq\lambda_{l}\left(A_{2}\right)\leq\lambda_{n_{1}-n_{2}+l}\left(A_{1}\right)$
for $l=1,2,\ldots,n_{2}$. From $l=k$ we get $\lambda_{k}\left(A_{1}\right)\leq\lambda_{k}\left(A_{2}\right)\leq\lambda_{k}\left(A_{3}\right)$,
and from $l=n_{2}-n_{3}+k$ we get $\lambda_{k}\left(A_{3}\right)\leq\lambda_{n_{2}-n_{3}+k}\left(A_{2}\right)\leq\lambda_{n_{1}-n_{3}+k}\left(A_{1}\right)$,
such that $\lambda_{k}\left(A_{1}\right)\leq\lambda_{k}\left(A_{3}\right)\leq\lambda_{n_{1}-n_{3}+k}\left(A_{1}\right)$
for $k=1,2,\ldots,n_{3}$ and we obtain that $A_{3}\propto A_{1}$.\end{proof}

The most commonly studied matrices in algebraic graph theory are the
\emph{adjacency matrix} $A\left(\G\right)\in\R^{\left|\V\right|\times\left|\V\right|}$,
the \emph{Laplacian matrix} $L\left(\G\right)\in\R^{\left|\V\right|\times\left|\V\right|}$
and the \emph{normalized Laplacian matrix} $\mathcal{L}\left(\G\right)\in\R^{\left|\V\right|\times\left|\V\right|}$,
all of which are real symmetric matrices. They are defined below,
where each row and column is indexed by a vertex in the graph $\mathcal{G}$
\cite{Godsil2001}, 
\[
[A(\mathcal{G})]_{uv}=\left\{ \begin{array}{cl}
1, & u\sim v\\
0, & \mbox{otherwise}
\end{array}\right.,
\]
\[
[L(\mathcal{G})]_{uv}=\left\{ \begin{array}{cl}
d_{u}\left(\G\right), & u=v\\
-1 & u\sim v\\
0, & \mbox{otherwise}
\end{array}\right.,
\]
and 
\[
[\mathcal{L}(\mathcal{G})]_{uv}=\left\{ \begin{array}{cl}
1, & u=v\\
-\left(\sqrt{d_{u}\left(\G\right)d_{v}\left(\G\right)}\right)^{-1} & u\sim v\\
0, & \mbox{otherwise}
\end{array}\right..
\]

We now extend the notion of spectral interlacing properties to graphs.

\begin{Definition}[interlacing graphs]\label{interlacing graphs}Consider
two graphs $\G_{n}$ and $\G_{r}$ of order $n$ and $r$ respectively,
with $n>r$, and let $M(\G)\in\R^{n\times n}$ be any real symmetric
matrix associated with the graph $\G$. We say that the two graphs
are \emph{$M$-interlacing }if $M\left(\G_{r}\right)\propto M\left(\G_{n}\right)$,
and denote the property by $\G_{r}\propto_{M}\G_{n}$.\end{Definition}

A problem arising naturally from the definition of interlacing graphs
is the interlacing graph reduction problem.

\begin{prob}[interlacing graph reduction]\label{interlacing graph reduction}
Consider a graph $\G_{n}$ of order $n$ and let $M(\G)\in\R^{n\times n}$
be any real symmetric matrix associated with the graph $\G$. Find
a graph $\G_{r}$ of a given order $r<n$ such that $\G_{r}\propto_{M}\G_{n}$.\end{prob}

Finding a solution to Problem \ref{interlacing graph reduction} may
be numerically intractable for a moderate number of nodes, as the
number $c_{r}$ of simple connected graphs of order $r$ increases
exponentially according to the recurrence $\sum_{k}\binom{r}{k}kc_{k}2^{\binom{r-k}{2}}=r2^{\binom{r}{2}}$
for $r\geq1$ \citep[p.87]{wilf1994generatingfunctionology}, e.g.,
for $r=1,\ldots,6$, $c_{r}=1,\,1,\,4,\,38,\,728,\,26704$.

A powerful tool for proving interlacing results is the Courant-Fischer
theorem, e.g., that a symmetric matrix and a principle submatrix of
that matrix interlace \cite{haemers1995interlacing}, which leads
to an adjacency interlacing theorem for node-removal graph reductions:

\begin{theorem}[Adjacency interlacing node-removal]\label{node removal contraction interlacing-1}
Consider a graph $\G$ and a node subset $\V_{S}\subset\V\left(\G\right)$.
Then $\G\backslash\V_{S}\propto_{A}\G$.\end{theorem}

\begin{proof}The matrix $A\left(\G\backslash\V_{S}\right)$ is a
principle submatrix of $A\left(\G\right)$, therefore, $\G\backslash\V_{S}\propto_{A}\G$.\end{proof}

Utilizing the Courant-Fischer theorem and the following $\min\max$
inequalities (Proposition \ref{minmax properties}) an interlacing
graph reduction theorem is derived. We first introduce some notations
to simplify the statement. A $k$-dimensional subspace of $\R^{n}$
is denoted as $\mathcal{F}_{n}^{\left(k\right)}$. For an $r$-dimensional
subspace $\mathcal{F}_{n}^{(r)}$, we define the injective map $p_{\mathcal{F}_{n}^{\left(r\right)}}:\R^{r}\rightarrow\R^{n}$,
such that $x\in\R^{r}\mapsto y\in\mathcal{F}_{n}^{(r)}$.

\begin{theorem}[Courant-Fischer \cite{Horn1991}]\label{Courant-Fischer}Consider
a real symmetric matrix $M\in\mathbb{R}^{n\times n}$, then for $k\in\left[1,n\right]$
\begin{equation*}
\lambda_{k}\left(M\right)=\underset{\mathcal{F}_{n}^{\left(n-k+1\right)}}{\max}\;\underset{\substack{x\in\mathcal{F}_{n}^{\left(n-k+1\right)}\\
x\neq0
}
}{\min}R\left(M,x\right),
\end{equation*}
and
\begin{equation*}
\lambda_{k}\left(M\right)=\underset{\mathcal{F}_{n}^{\left(k\right)}}{\min}\;\underset{\substack{x\in\mathcal{F}_{n}^{\left(k\right)}\\
x\neq0
}
}{\max}R\left(M,x\right),
\end{equation*}
where $R\left(M,x\right)\triangleq\frac{x^{T}Mx}{x^{T}x}$ is the
Rayleigh quotient. \end{theorem}

\begin{prop}\label{minmax properties}Consider a subspace $\mathcal{F}_{n}^{(r)}$
for $r<n$, and let $f\left(x\right):\R^{n}\rightarrow\R$ be a real-valued
function that attains a minimum and a maximum on $\R^{n}\backslash\left\{ 0\right\} $.
Then the following holds for $k\in\left[1,r\right]$:
\begin{itemize}
\item[i)] $\underset{\mathcal{F}_{n}^{\left(n-k+1\right)}}{\max}\;\underset{\substack{x\in\mathcal{F}_{n}^{\left(n-k+1\right)}\\
x\neq0
}
}{\min}f\left(x\right)\leq\max\limits _{\mathcal{F}_{r}^{\left(r-k+1\right)}}\;\min\limits _{\substack{\tilde{x}\in\mathcal{F}_{r}^{\left(r-k+1\right)}\\
\tilde{x}\neq0
}
}f\left(p_{\mathcal{F}_{n}^{\left(r\right)}}\left(\tilde{x}\right)\right)$,
\item[ii)] $\min\limits _{\mathcal{F}_{n}^{\left(n-r+k\right)}}\;\max\limits _{\substack{x\in\mathcal{F}_{n}^{\left(n-r+k\right)}\\
x\neq0
}
}f\left(x\right)\geq\min\limits _{\mathcal{F}_{r}^{\left(k\right)}}\;\max\limits _{\substack{\tilde{x}\in\mathcal{F}_{r}^{\left(k\right)}\\
\tilde{x}\neq0
}
}f\left(p_{\mathcal{F}_{n}^{\left(r\right)}}\left(\tilde{x}\right)\right)$.
\end{itemize}
\end{prop}

\begin{proof}
We first prove ($i$). Let $s\equiv n-k+1$. For all
$\mathcal{F}_{n}^{\left(s\right)}\subseteq\R^{n}$, 
\begin{align*}
\underset{\substack{x\in\mathcal{F}_{n}^{\left(s\right)}\\
x\neq0}}{\min}f\left(x\right) & = \min\left\{ \min\limits _{\substack{x\in\mathcal{F}_{n}^{\left(s\right)}\cap\mathcal{F}_{n}^{\left(r\right)}\\
x\neq0
}
}f\left(x\right), \; \min\limits _{\substack{x\in\mathcal{F}_{n}^{\left(s\right)}\backslash\left\{ \mathcal{F}_{n}^{\left(s\right)}\cap\mathcal{F}_{n}^{\left(r\right)}\right\} \\
x\neq0
}
}f\left(x\right)\right\} \nonumber \\
 & \leq\min\limits _{\substack{x\in\mathcal{F}_{n}^{\left(s\right)}\cap\mathcal{F}_{n}^{\left(r\right)}\\
x\neq0
}
}f\left(x\right),
\end{align*}
and we obtain that
\begin{equation}
\underset{\mathcal{F}_{n}^{\left(s\right)}}{\max}\;\underset{\substack{x\in\mathcal{F}_{n}^{\left(s\right)}\\
x\neq0
}
}{\min}f\left(x\right)\leq\max\limits _{\mathcal{F}_{n}^{\left(s\right)}}\;\min\limits _{\substack{x\in\mathcal{F}_{n}^{\left(s\right)}\cap\mathcal{F}_{n}^{\left(r\right)}\\
x\neq0
}
}f\left(x\right).\label{minmax inequality 1}
\end{equation}
Since $k\leq r$, then $s=n-k+1>n-r$ and 
\begin{equation*}
\dim\left(\mathcal{F}_{n}^{\left(s\right)}\cap\mathcal{F}_{n}^{\left(r\right)}\right)\geq s-\left(n-r\right).
\end{equation*}
Therefore, 
\begin{align*}
\max\limits _{\mathcal{F}_{n}^{\left(s\right)}}\;\min\limits _{\substack{x\in\mathcal{F}_{n}^{\left(s\right)}\cap\mathcal{F}_{n}^{\left(r\right)}\\
x\neq0
}
}f\left(x\right)=
\max\limits _{\mathcal{F}_{n}^{\left(s-\left(n-r\right)\right)}\subseteq\mathcal{F}_{n}^{\left(r\right)}}\;\min\limits _{\substack{
x\in\mathcal{F}_{n}^{\left(s-\left(n-r\right)\right)}\\
x\neq0}}f\left(x\right).
\end{align*}

For each $\mathcal{F}_{n}^{\left(s-\left(n-r\right)\right)}\subseteq\mathcal{F}_{n}^{\left(r\right)}$
we can find $\mathcal{\tilde{F}}_{r}^{\left(s-\left(n-r\right)\right)}\subseteq\R^{r}$
that is mapped to it by $p_{\mathcal{F}_{n}^{\left(r\right)}}\left(\tilde{x}\right)$,
\begin{equation*}
\mathcal{\tilde{F}}_{r}^{\left(s-\left(n-r\right)\right)}=\left\{ \tilde{x}\in\R^{r}|p_{\mathcal{F}_{n}^{\left(r\right)}}\left(\tilde{x}\right)\in\mathcal{F}_{n}^{\left(s-\left(n-r\right)\right)}\right\} ,
\end{equation*}
such that 
\begin{align*}
\min\limits _{\substack{x\in\mathcal{F}_{n}^{\left(s-\left(n-r\right)\right)}\\
x\neq0
}}f\left(x\right)=\min\limits _{\substack{
\tilde{x}\in\tilde{\mathcal{F}}_{r}^{\left(s-\left(n-r\right)\right)}\\
\tilde{x}\neq0}}f\left(p_{\mathcal{F}_{n}^{\left(r\right)}}\left(\tilde{x}\right)\right).
\end{align*}

Maximizing over all $\mathcal{F}_{n}^{\left(s-\left(n-r\right)\right)}\subseteq\mathcal{F}_{n}^{\left(r\right)}$
we obtain 
\begin{align*}
\max\limits _{\mathcal{F}_{n}^{\left(s-\left(n-r\right)\right)}\subseteq\mathcal{F}_{n}^{\left(r\right)}}\;\min\limits _{\substack{
x\in\mathcal{F}_{n}^{\left(s-\left(n-r\right)\right)}\\
x\neq0
}}f\left(x\right) & =\max\limits _{\mathcal{F}_{r}^{\left(s-\left(n-r\right)\right)}}\;\min\limits _{\substack{
\tilde{x}\in\tilde{\mathcal{F}}_{r}^{\left(s-\left(n-r\right)\right)}\\
\tilde{x}\neq0}}f\left(p_{\mathcal{F}_{n}^{\left(r\right)}}\left(\tilde{x}\right)\right),
\end{align*}
and
\begin{equation}
\max\limits _{\mathcal{F}_{n}^{\left(n-k+1\right)}}\;\min\limits _{\substack{x\in\mathcal{F}_{n}^{\left(n-k+1\right)}\cap\mathcal{F}_{n}^{\left(r\right)}\\
x\neq0
}
}f\left(x\right)=\max\limits _{\mathcal{F}_{r}^{\left(r-k+1\right)}}\;\min\limits _{\substack{
\tilde{x}\in\tilde{\mathcal{F}}_{r}^{\left(r-k+1\right)}\\
\tilde{x}\neq0}}f\left(p_{\mathcal{F}_{n}^{\left(r\right)}}\left(\tilde{x}\right)\right).
\label{minmax equality 1}
\end{equation}
Equation \eqref{minmax equality 1} together with \eqref{minmax inequality 1} completes the proof of ($i$).

The proof of (ii) is as follows. Let $s\equiv n-r+k$ . For all $\mathcal{F}_{n}^{\left(s\right)}\subseteq\R^{n}$,
\begin{align*}
\max\limits _{\substack{
x\in\mathcal{F}_{n}^{\left(s\right)}\\
x\neq0}}f\left(x\right) & = \max\left\{ \max\limits _{\substack{x\in\mathcal{F}_{n}^{\left(s\right)}\cap\mathcal{F}_{n}^{\left(r\right)}\\
x\neq0
}
}f\left(x\right), \; \max\limits _{\substack{x\in\mathcal{F}_{n}^{\left(s\right)}\backslash\left\{ \mathcal{F}_{n}^{\left(s\right)}\cap\mathcal{F}_{n}^{\left(r\right)}\right\} \\
x\neq0
}
}f\left(x\right)\right\} \nonumber \\
 & \geq\max\limits _{\substack{x\in\mathcal{F}_{n}^{\left(s\right)}\cap\mathcal{F}_{n}^{\left(r\right)}\\
x\neq0
}
}f\left(x\right),
\end{align*}
and 
\begin{equation}
\min\limits _{\mathcal{F}_{n}^{\left(s\right)}}\;\max\limits _{\substack{x\in\mathcal{F}_{n}^{\left(s\right)}\\
x\neq0
}
}f\left(x\right)\geq\min\limits _{\mathcal{F}_{n}^{\left(s\right)}}\;\max\limits _{\substack{x\in\mathcal{F}_{n}^{\left(s\right)}\cap\mathcal{F}_{n}^{\left(r\right)}\\
x\neq0
}
}f\left(x\right). \label{minmax inequality 2}
\end{equation}

Since $k\geq1$ then $s=n-r+k>n-r$ and 
\begin{equation*}
\dim\left(\mathcal{F}_{n}^{\left(s\right)}\cap\mathcal{F}_{n}^{\left(r\right)}\right)\geq s-\left(n-r\right),
\end{equation*}
 and we can then replace $\max\min$ with $\min\max$ in the above
proof of (i) and obtain 
\begin{equation}
\min\limits _{\mathcal{F}_{n}^{\left(n-r+k\right)}}\;\max\limits _{\substack{x\in\mathcal{F}_{n}^{\left(n-r+k\right)}\cap\mathcal{F}_{n}^{\left(r\right)}\\
x\neq0
}
}f\left(x\right)=\min\limits _{\mathcal{F}_{r}^{\left(k\right)}}\;\max\limits _{\substack{\tilde{x}\in\mathcal{F}_{r}^{\left(k\right)}\\
\tilde{x}\neq0
}
}f\left(p_{\mathcal{F}_{n}^{\left(r\right)}}\left(\tilde{x}\right)\right).
\label{minmax equality 2} 
\end{equation}
Equation \eqref{minmax equality 2} together with \eqref{minmax inequality 2} completes the proof of ($ii$).
\end{proof}

\begin{theorem}[Interlacing graph reduction theorem]\label{interlacing graph reduction theorem}
Consider two graphs $\G_{n}$ and $\G_{r}$ of order $n$ and $r$
respectively, with $n>r$, and let $M(\G)\in\R^{n\times n}$ be any
real symmetric matrix associated with the graph $\G$. If there exists
$r$-dimensional subspaces $\mathcal{A},\mathcal{B}\subseteq\R^{n}$
such that $\forall x\in\R^{r}\backslash\left\{ 0\right\} $,
\begin{equation*}
R\left(M\left(\G_{n}\right),p_{\mathcal{A}}\left(x\right)\right)\leq R\left(M\left(\G_{r}\right),x\right),
\end{equation*}
and
\begin{equation*}
R\left(M\left(\G_{n}\right),p_{\mathcal{B}}\left(x\right)\right)\geq R\left(M\left(\G_{r}\right),x\right),
\end{equation*}
then $\G_{r}\propto_{M}\G_{n}$.\end{theorem}

\begin{proof}In order for $\G_{n}$ and $\G_{r}$ to be $M$-interlacing
(Definition \ref{minmax properties}) we must prove that $\lambda_{k}\left(M\left(\G_{n}\right)\right)\leq\lambda_{k}\left(M\left(\G_{r}\right)\right)\leq\lambda_{n-r+k}\left(M\left(\G_{n}\right)\right)$
for $k\in\left[1,r\right]$. From the Courant--Fischer theorem (Theorem
\ref{Courant-Fischer}) we have 
\begin{equation*}
\lambda_{k}\left(M\left(\G_{n}\right)\right)=\underset{\mathcal{F}_{n}^{\left(n-k+1\right)}}{\max}\;\underset{\substack{x\in\mathcal{F}_{n}^{\left(n-k+1\right)}\\
x\neq0
}
}{\min}R\left(M\left(\G_{n}\right),x\right),
\end{equation*}
and from the min-max properties (Proposition \ref{minmax properties})
with $\mathcal{F}_{n}^{\left(r\right)}\equiv\mathcal{A}$ we have
for $k\in\left[1,r\right]$,
\begin{equation*}
\lambda_{k}\left(M\left(\G_{n}\right)\right)\leq\underset{\mathcal{F}_{r}^{\left(r-k+1\right)}}{\max}\;\underset{\substack{x\in\mathcal{F}_{r}^{\left(r-k+1\right)}\\
x\neq0
}
}{\min}R\left(M\left(\G_{n}\right),p_{\mathcal{A}}\left(x\right)\right).
\end{equation*}

Since $R\left(M\left(\G_{n}\right),p_{\mathcal{A}}\left(x\right)\right)\leq R\left(M\left(\G_{r}\right),x\right)$,
therefore, 
\begin{align*}
\lambda_{k}\left(M\left(\G_{n}\right)\right) & \leq\underset{\mathcal{F}_{r}^{\left(r-k+1\right)}}{\max}\;\underset{\substack{x\in\mathcal{F}_{r}^{\left(r-k+1\right)}\\
x\neq0
}
}{\min}R\left(M\left(\G_{r}\right),x\right)\nonumber \\
 & =\lambda_{k}\left(M\left(\G_{r}\right)\right),
\end{align*}
and $\lambda_{k}\left(M\left(\G_{n}\right)\right)\leq\lambda_{k}\left(M\left(\G_{r}\right)\right)$
for $k\in\left[1,r\right]$. In order to complete the interlacing
proof it is left to show that $\lambda_{k}\left(M\left(\G_{r}\right)\right)\leq\lambda_{n-r+k}\left(M\left(\G_{n}\right)\right)$
for $k\in\left[1,r\right]$. From the Courant--Fischer theorem (Theorem
\ref{Courant-Fischer}) we get 
\begin{equation*}
\lambda_{n-r+k}\left(M\left(\G_{n}\right)\right)=\underset{\mathcal{F}_{n}^{\left(n-r+k\right)}}{\min}\;\underset{\substack{x\in\mathcal{F}_{n}^{\left(n-r+k\right)}\\
x\neq0
}
}{\max}R\left(M\left(\G_{n}\right),x\right),
\end{equation*}
and from the min-max properties (Proposition \ref{minmax properties})
with $\mathcal{F}_{n}^{\left(r\right)}\equiv\mathcal{B}$ we have
\begin{equation*}
\lambda_{n-r+k}\left(M\left(\G_{n}\right)\right)\geq\underset{\mathcal{F}_{r}^{\left(k\right)}}{\min}\;\underset{\substack{x\in\mathcal{F}_{r}^{\left(k\right)}\\
x\neq0
}
}{\max}R\left(M\left(\G_{n}\right),p_{\mathcal{B}}\left(x\right)\right).
\end{equation*}

Since $R\left(M\left(\G_{n}\right),p_{\mathcal{B}}\left(x\right)\right)\geq R\left(M\left(\G_{r}\right),x\right)$,
therefore, 
\begin{align*}
\lambda_{n-r+k}\left(M\left(\G_{n}\right)\right) & \geq\underset{\mathcal{F}_{r}^{\left(k\right)}}{\min}\;\underset{\substack{x\in\mathcal{F}_{r}^{\left(k\right)}\\
x\neq0
}
}{\max}R\left(M\left(\G_{r}\right),x\right)\nonumber \\
 & =\lambda_{k}\left(M\left(\G_{r}\right)\right),
\end{align*}
and $\lambda_{k}\left(M\left(\G_{n}\right)\right)\leq\lambda_{n-r+k}\left(M\left(\G_{r}\right)\right)$
for $k\in\left[1,r\right]$, completing the proof.\end{proof}

In the next section we describe graph contractions as a constructive
method for performing graph reductions, and introduce a class of contractions
that, based on Theorem \ref{interlacing graph reduction theorem},
will lead to efficient algorithms for finding interlacing graph reductions.

\section{Graph Contractions\label{sec:Graph-Contractions}}

Graph contractions are a graph reduction method based on partitions
of the vertex set. They are a useful algorithmic tool applied to a
variety of graph-theoretical problems, e.g., for obtaining the connected
components \cite{chong1995finding} or finding all spanning trees
of a graph \cite{minty1965simple,winter1986algorithm}. We now define
several graph operations required for vertex partitions and graph
contractions and derive results that will allow us to relate graph
contractions and graph interlacing.

For an integer $r$ satisfying $1\leq r\leq n$, an \emph{$r$-partition}
of a vertex set $\V$ of order $n$, denoted $\pi_{r}\left(\V\right)$,
is a set of $r$ cells $\left\{ C_{i}\right\} _{i=1}^{r}$ such that
$C_{i}\cap C_{j}=\emptyset$ and $\cup_{i=1}^{r}C_{i}=\V$. We denote
the $i$th cell of a partition $\pi$ as $C_{i}\left(\pi\right)$,
and the \emph{cell neighborhood }$\N_{C_{i}}\left(\G\right)$ is defined
as $\N_{C_{i}}\triangleq\left\{ \cup_{v\in C_{i}}\N_{v}\left(\G\right)\right\} \backslash C_{i}$.
For $r=n$, $C_{i}\left(\pi_{n}\right)=i$ is the \textit{identity
partition}, which contains $n$ singletons (a cell with a single vertex).
An \textit{atom partition} $\pi_{n-1}\left(\V\right)$ contains $n-2$
singletons and a single 2-vertex cell. The set of all $r$-{partitions}
of $\V$ is denoted by $\Pi_{r}\left(\V\right)$, and the set of all
partitions of $\V$ is $\Pi\left(\V\right)\triangleq\cup_{r=1}^{n}\Pi_{r}\left(\V\right)$.
For a graph $\graph$, we may denote $\pi_{r}\left(\V\right)$ and
$\Pi_{r}\left(\V\right)$ as $\pi_{r}\left(\G\right)$ and $\Pi_{r}\left(\G\right)$.
For a graph with $n_{cc}$ connected components, we define the \emph{connected
components partition} $\pi_{cc}\left(\G\right)$ as the partition
$\pi_{cc}\left(\G\right)=\left\{ C_{i}\right\} _{i=1}^{n_{cc}}$,
such that $\G\left[C_{i}\right]$ is the $i$th connected component
of $\G$. Hereafter $\graph$ is a simple connected graph of order
$n$.

\begin{Definition}[partition function]\label{Partition Function}
For a graph $\G$ and $r$-partition $\pi\in\Pi_{r}\left(\G\right)$,
the \emph{partition function} is a map $f_{\pi}:\V\left(\G\right)\rightarrow\left[1,r\right]$
from each node in $\V$ to its cell index, i.e., $f_{\pi}\left(v\right)\triangleq\left\{ i\in\left[1,r\right]|C_{i}\left(\pi\right)\cap v\neq\emptyset\right\} $.
More generally, for a subset $\V_{S}\subseteq\V\left(\G\right)$ we
have $f_{\pi}\left(\V_{S}\right)\triangleq\left\{ i\in\left[1,r\right]|C_{i}\left(\pi\right)\cap\mathcal{V}_{S}\neq\emptyset\right\} $.
\end{Definition}

The \textit{quotient} of a graph $\G$ over a partition $\pi\in\Pi_{r}(\G)$,
denoted by $\G/\pi$, is the multi-graph of order $r$ with an edge
$\left\{ u,v\right\} $ for each edge between nodes in $C_{u}\left(\pi\right)$
and $C_{v}\left(\pi\right)$, i.e., $\G/\pi=\left(\left[1,r\right],\left\{ \tilde{\epsilon}_{j}\right\} _{j=1}^{\left|\E\right|}\right)$
with 
$
\tilde{\epsilon}_{j}=\left\{ f_{\pi}\left(h_{\E}\left(\epsilon_{j}\right)\right),f_{\pi}\left(t_{\E}\left(\epsilon_{j}\right)\right)\right\} ,
$
where $\epsilon_{j}\in\E\left(\G\right)$ and $h_{\E}\left(\epsilon\right),t_{\E}\left(\epsilon\right):\E\left(\G\right)\rightarrow\V\left(\G\right)$
assign a head and a tail to the end-nodes of each edge (thus, $\epsilon=(h_{\E}(\epsilon),t_{\E}(\epsilon))$).
The \emph{graph contraction} of $\G$ over $\pi$ is the simple graph
denoted as $\gcont$ which is obtained from the quotient $\G/\pi$
by removing all self-loops and redundant duplicate edges. Equivalently,
$\gcont=\left(\left[1,r\right],\E_{r}\right)$ with $\E_{r}=\left\{ \tilde{\epsilon}\in\left[1,r\right]^{2}|\tilde{\epsilon}\in\E\left(\G/\pi\right),h_{\E}\left(\tilde{\epsilon}\right)\neq t_{\E}\left(\tilde{\epsilon}\right)\right\} $.
If $\pi$ is an atom partition we call $\G\sslash\pi$ an \textit{atom
contraction}. For example, consider the partition of $\pi=\left\{ \left\{ v_{1}\right\} ,\left\{ v_{2}\right\} ,\left\{ v_{3}\right\} ,\left\{ v_{4},v_{5}\right\} \right\} $,
for the graph $\G$ shown in Figure \ref{fig:Quotient-and-graph}.
The quotient $\G/\pi$ and contraction $\gcont$ of the graph are
shown in Figure \ref{fig:Quotient-and-graph}. Notice that this is
an example of an atom partition and atom contraction. 
\begin{center}
\begin{figure}
\centering{}\hfill{}\subfloat[Full order graph $\protect\G$ and its vertex partition $\pi$.]{\centering{}\includegraphics[width=1.4in,height=2in]{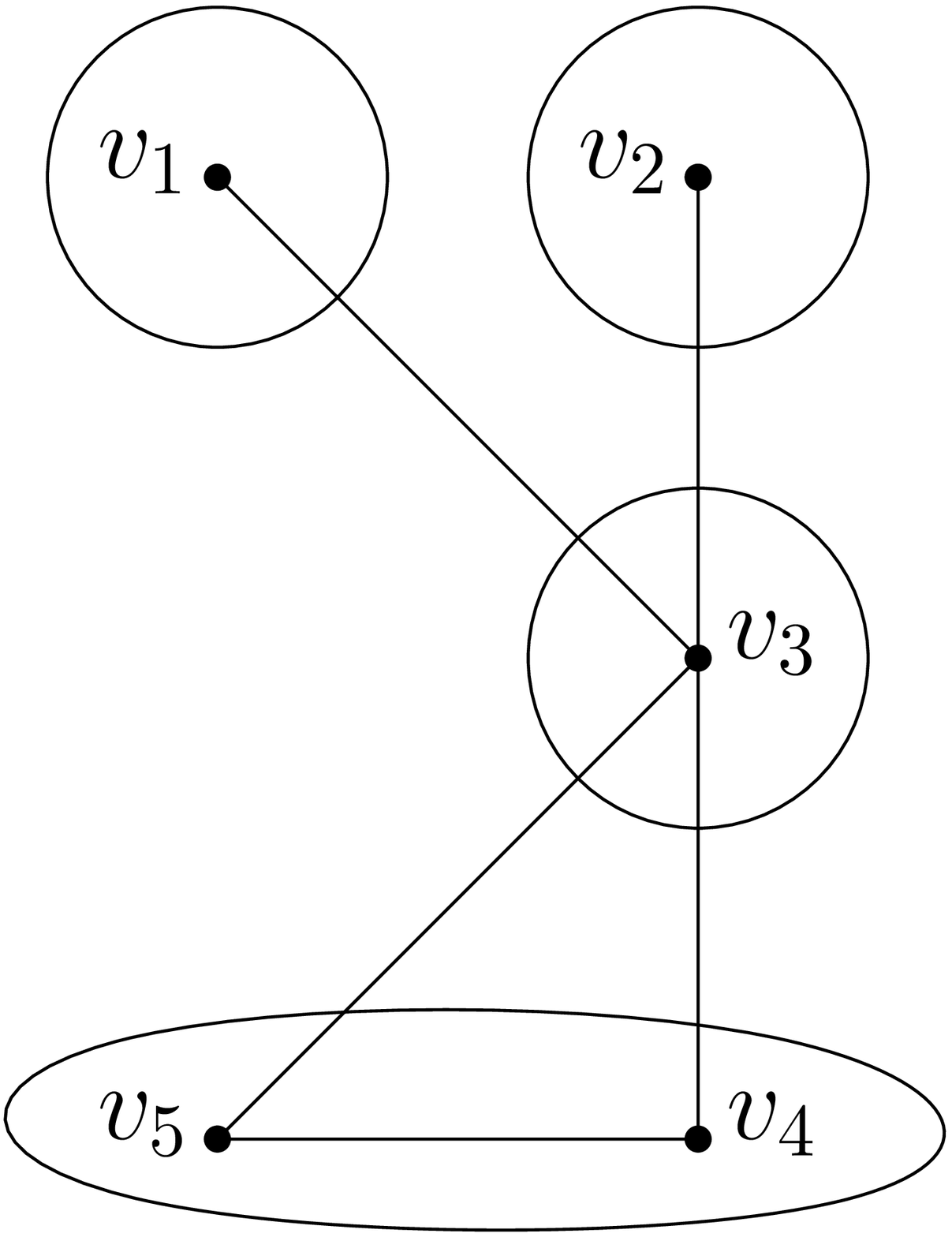}}\hfill{}\subfloat[The graph quotient $\protect\G\backslash\pi$.]{\centering{}\includegraphics[width=1.4in,height=2in]{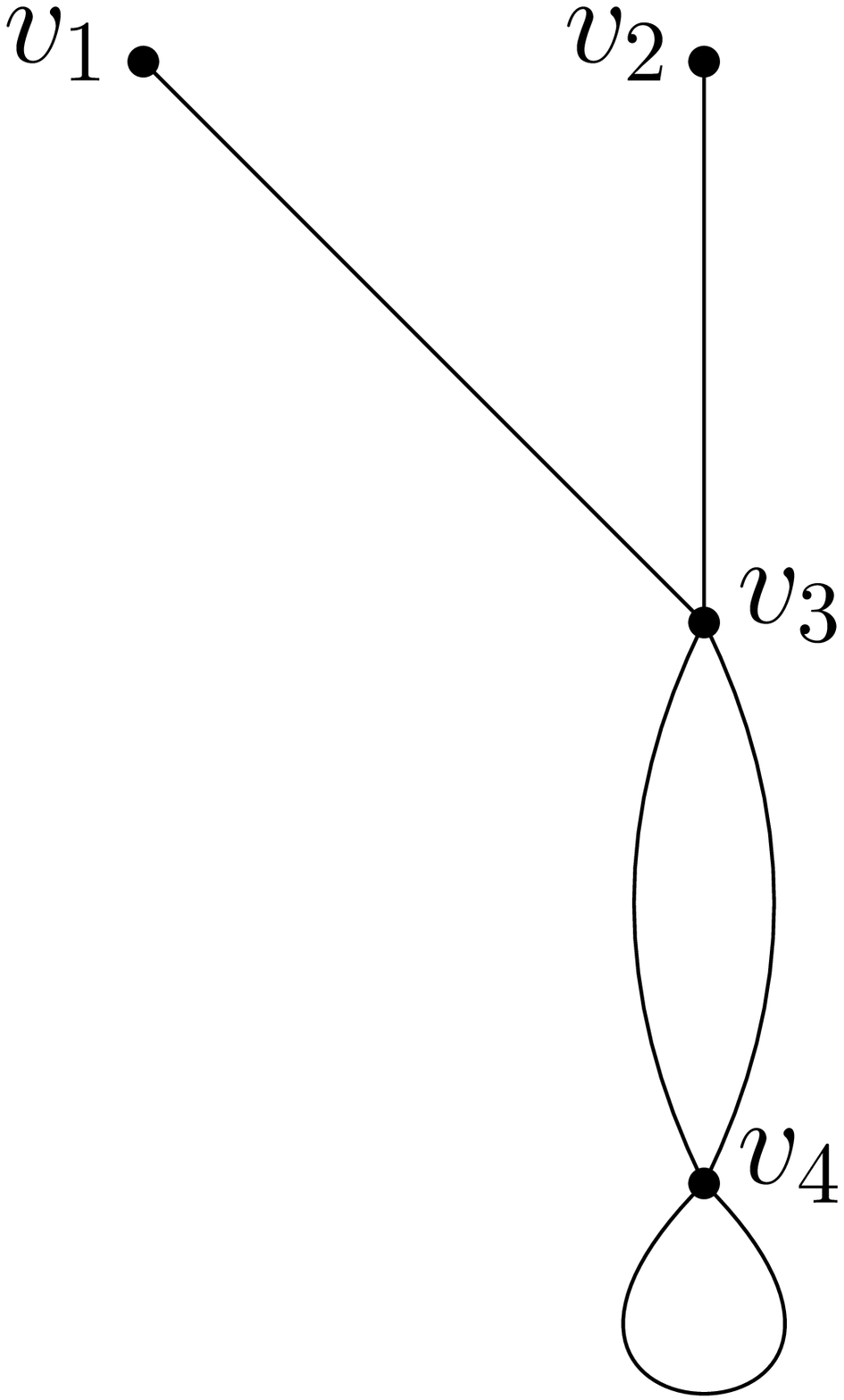}}\hfill{}\subfloat[The graph contration $\protect\gcont$.]{\centering{}\includegraphics[width=1.4in,height=2in]{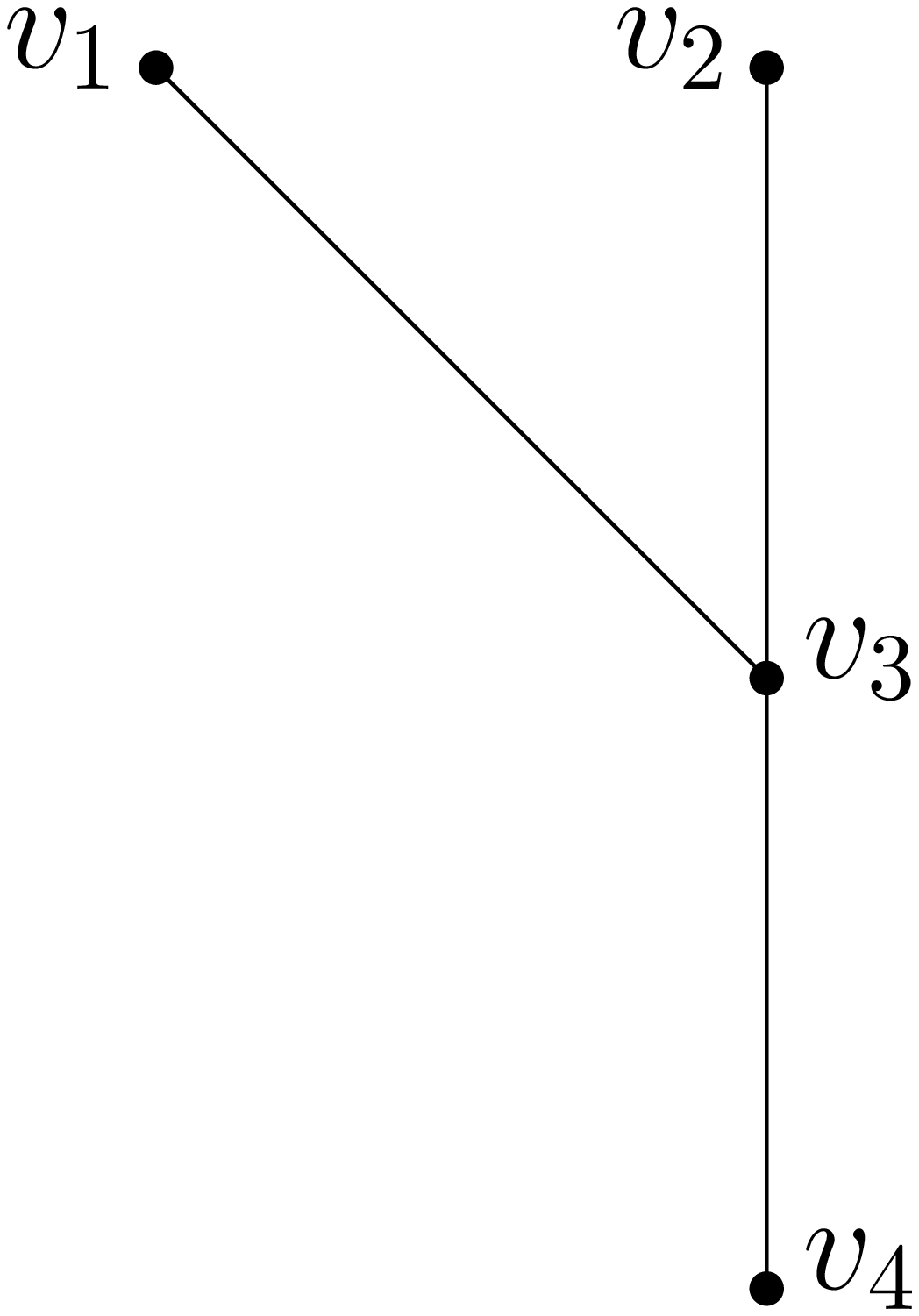}}\hfill{}\caption{Full order graph and its quotient and contraction over the vertex
partition $\pi=\left\{ \left\{ v_{1}\right\} ,\left\{ v_{2}\right\} ,\left\{ v_{3}\right\} ,\left\{ v_{4},v_{5}\right\} \right\} $.\label{fig:Quotient-and-graph}}
\end{figure}
\end{center}

Node removal is the simplest graph-reduction method. However, in some
cases the same reduced graph can be obtained either from node-removal
or from a graph contraction. We define here these contractions as
node-removal equivalent contractions.

\begin{Definition}[node-removal equivalent contraction]\label{node-removal equivalent contraction}
For the graph $\G$ and its contraction $\gcont$, we say that $\gcont$
is \emph{node-removal equivalent} if there is a subset $\V_{S}\subset\V\left(\G\right)$
such that $\G\sslash\pi=\G\backslash\V_{S}$.\end{Definition}

Cycles play an important role in the properties of graphs, and we
define a cycle-invariant graph contraction as a contraction that preserves
the cycle structure of the full graph.

\begin{Definition}[cycle-invariant contraction]\label{cycle-invariat contraction}
Consider a graph $\G$ and its contraction $\gcont$, then we say
that the contraction $\gcont$ is \emph{cycle-invariant} if there
is one-to-one mapping between the set of simple cycles of the full-order
graph and the set of simple cycles of the contracted graph.\end{Definition}

For example, consider the partition $\pi=\left\{ \left\{ v_{1},v_{2},v_{3}\right\} ,\left\{ v_{4}\right\} ,\left\{ v_{5}\right\} \right\} $
for the graph shown in Figure \ref{fig:Cycle-invariant-contraction}.
The resulting contraction over the graph is cycle-invariant (Definition
\ref{cycle-invariat contraction}) with the cycle $v_{3}v_{4}v_{5}v_{3}$
of $\G$ mapped to the cycle $v_{1}v_{2}v_{3}v_{1}$ of $\gcont$,
and is also node-removal equivalent (Definition \ref{node-removal equivalent contraction})
with $\V_{S}=\left\{ v_{1},v_{2}\right\} $. Notice that if the edge
$\{v_{1},v_{5}\}$ were added in Figure \ref{fig:Cycle-invariant-contraction},
the same contraction would \emph{not} be a cycle-invariant contraction;
however, it would still be node-removal equivalent with $\V_{S}=\left\{ v_{1},v_{2}\right\} $. 
\begin{center}
\begin{figure}
\hfill{}\subfloat[Full order graph $\protect\G$ and its vertex partition $\pi=\left\{ \left\{ v_{1},v_{2},v_{3}\right\} ,\left\{ v_{4}\right\} ,\left\{ v_{5}\right\} \right\} $.]{\centering{}\includegraphics[width=2in,height=2.5in]{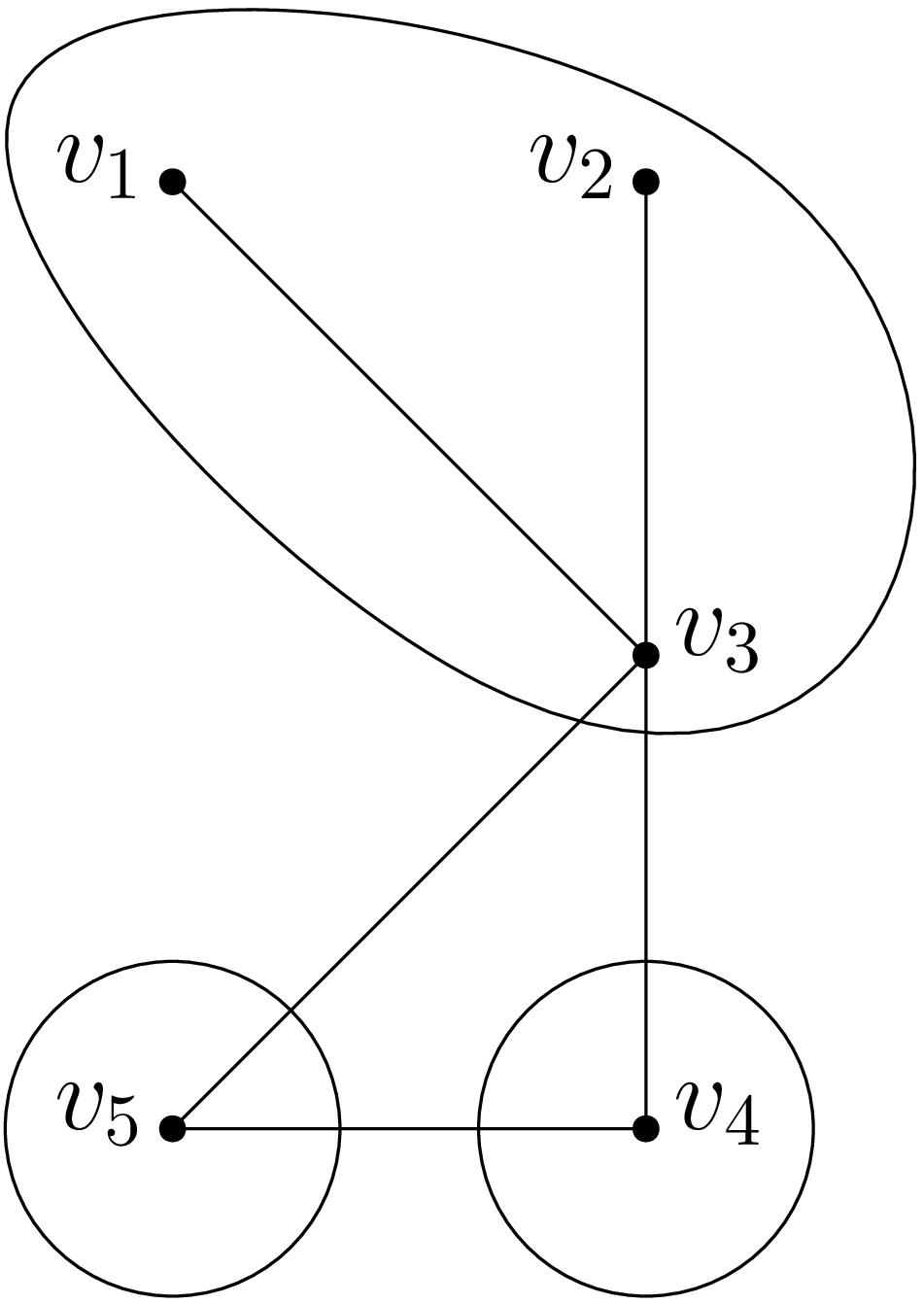}}\hfill{}\subfloat[Cycle invariant graph and node-removal equivalent contraction $\protect\gcont$.]{\centering{}\includegraphics[width=2in,height=2.5in]{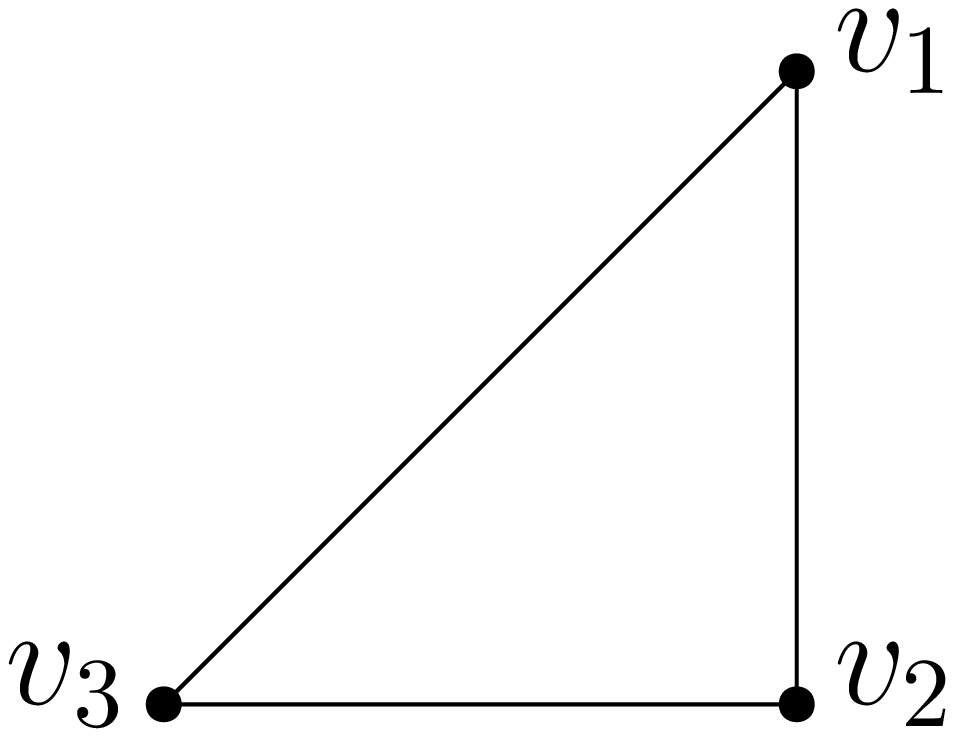}}\hfill{}\caption{Full order graph and its cycle-invariant and node-removal equivalent
contraction.\label{fig:Cycle-invariant-contraction}}
\end{figure}
\end{center}

\begin{lem}[subgraph contraction lemma] \label{subgraph contraction}Consider
a graph $\G$ and its subgraph $\G_{R}=\G\backslash\E_{R}$ for $\E_{R}\subseteq\E\left(\G\right)$.
Then for any $\pi\in\Pi\left(\G\right)$, $\G_{R}\sslash\pi\subseteq\gcont$.
\end{lem}

\begin{proof}For any $\tilde{\epsilon}\in\E\left(\G_{R}\sslash\pi\right)$
we can find $\epsilon\in\E\left(\G_{R}\right)$ such that \\
$\tilde{\epsilon}=\left\{ f_{\pi}(h_{\E}\left(\epsilon\right)),f_{\pi}\left(t_{\E}\left(\epsilon\right)\right)\right\} $.
Since $\E\left(\G_{R}\right)\subseteq\E\left(\G\right)$, therefore
$\epsilon\in\E\left(\G\right)$ and $\left\{ f_{\pi}(h_{\E}\left(\epsilon\right)),f_{\pi}\left(t_{\E}\left(\epsilon\right)\right)\right\} \in\E\left(\G\sslash\pi\right)$.
We conclude that $\E\left(\G_{R}\sslash\pi\right)\subseteq\E\left(\G\sslash\pi\right)$,
and since $\V\left(\G_{R}\sslash\pi\right)=\V\left(\G\sslash\pi\right)$
we obtain that $\G_{R}\sslash\pi\subseteq\G\sslash\pi$.\end{proof}

\begin{lem}\label{contracted neigborhood connectedness} Consider
a graph $\G$ and its contraction $\gcont$ for $\pi\in\Pi\left(\G\right)$.
Then $\forall u\in\V\left(\G\right),\forall\tilde{u}\in\V\left(\gcont\right)$,
we have $u\in\N_{C_{\tilde{u}}}\left(\G\right)$ if and only if $f_{\pi}\left(u\right)\sim\tilde{u}$.
\end{lem} \begin{proof} If $u\in\N_{C_{\tilde{u}}}$ then $\exists v\in C_{\tilde{u}}$
such that $u\sim v$ with $\epsilon=\left\{ u,v\right\} \in\E\left(\G\right)$,
and therefore $\left\{ f_{\pi}\left(u\right),f_{\pi}\left(v\right)\right\} =\left\{ f_{\pi}\left(u\right),\tilde{u}\right\} \in\E\left(\gcont\right)$
and $f_{\pi}\left(u\right)\sim\tilde{u}$. If $f_{\pi}\left(u\right)\sim\tilde{u}$,
then $\exists v\in C_{\tilde{u}}$ such that $u\sim v$ and therefore
$u\in\N_{C_{\tilde{u}}}$. \end{proof}

\begin{lem} \label{If--is connected}If a graph $\G$ is connected
then its graph contraction $\G\sslash\pi$ is connected. \end{lem}

\begin{proof}If $\G$ is connected then $\forall u,v\in\V$, there
is a path $uu_{1}u_{2}\ldots u_{p}v$. For any $\tilde{u},\tilde{v}\in\V\left(\G\sslash\pi\right)$
we can find $u,v\in\V$ such that $f_{\pi}\left(u\right)=\tilde{u}$
and $f_{\pi}\left(v\right)=\tilde{v}$. If we then apply the partition
function on the path $uu_{1}u_{2}\ldots u_{p}v$ we obtain a walk
(including self loops) in $\gcont$, $\tilde{u}f_{\pi}\left(u_{1}\right)f_{\pi}\left(u_{2}\right)\ldots f_{\pi}\left(u_{p}\right)\tilde{v}$,
therefore, $\gcont$ is a connected graph.\end{proof}

The following result relates the degree of a node in a contracted
graph to its cell-neighborhood.

\begin{prop}[degree-contraction]\label{degree-contraction}Consider
a graph $\G$ and its contraction $\gcont$ for $\pi\in\Pi\left(\G\right)$.
Then $\forall\tilde{v}\in\V\left(\gcont\right)$, $d_{\tilde{v}}\left(\gcont\right)=\left|f_{\pi}\left(\N_{C_{\tilde{v}}}\left(\G\right)\right)\right|$.\end{prop}

\begin{proof}From Definition \ref{Partition Function} we have $f_{\pi}\left(\N_{C_{\tilde{v}}}\left(\G\right)\right)=\left\{ i\in\left[1,r\right]|C_{i}\left(\pi\right)\cap\N_{C_{\tilde{v}}}\left(\G\right)\neq\emptyset\right\} $,
and from Lemma \ref{contracted neigborhood connectedness} we obtain
that $\forall\tilde{u},\tilde{v}\in\V\left(\gcont\right)$, $\tilde{v}\sim\tilde{u}$
if and only if $\tilde{u}\in f_{\pi}\left(\N_{C_{\tilde{v}}}\right)$
such that $f_{\pi}\left(\N_{C_{\tilde{v}}}\right)=\N_{\tilde{v}}\left(\gcont\right)$,
and therefore, $d_{\tilde{v}}\left(\gcont\right)=\left|f_{\pi}\left(\N_{C_{\tilde{v}}}\right)\right|$.
\end{proof}

\subsection{Graph Contraction Posets}

Partially-ordered sets (posets) are an essential set-theoretical concept.
Chains are totally-ordered subsets of the posets and are a useful
tool for proving set-theoretical results. Here we show how graph contractions
fall under the definition of a poset and will then establish contraction
chains and their corresponding contraction sequences as a basis for
proving cases of graph matrices interlacing.

Two partitions $\pi_{r_{1}},\pi_{r_{2}}\in\Pi\left(\V\right)$ may
comply with a refinement relation.

\begin{Definition}[refinement]\label{partition refinement} Consider
two partitions $\pi_{r_{1}},\pi_{r_{2}}\in\Pi\left(\V\right)$ of
a vertex set $\V$ where $r_{1}\leq r_{2}\leq\left|\V\right|$. Then
we say $\pi_{r_{2}}$ is a \emph{refinement} of $\pi_{r_{1}}$ if
$\forall j\in\left\{ 1,2,\ldots,r_{2}\right\} $ we can find $i\in\left\{ 1,2,\ldots,r_{1}\right\} $
such that $C_{j}\left(\pi_{r_{2}}\right)\subseteq C_{i}\left(\pi_{r_{1}}\right)$,
and we denote $\pi_{r_{2}}\leq\pi_{r_{1}}$. If $\pi_{r_{2}}\leq\pi_{r_{1}}$
and $r_{1}<r_{2}$ we denote $\pi_{r_{2}}<\pi_{r_{1}}$. An $N$\emph{-chain}
is a partition set $\chi\left(\V\right)=\left\{ \pi_{r_{i}}\right\} _{i=1}^{N}\subseteq\Pi\left(\V\right)$
such that $\pi_{r_{1}}<\pi_{r_{2}}<\ldots<\pi_{r_{N}}$.\end{Definition}
If two partitions $\pi_{r_{1}},\pi_{r_{2}}\in\Pi\left(\V\right)$
comply with the refinement relation, we can construct the \emph{coarsening}
partition $\delta\left(\pi_{r_{2}},\pi_{r_{1}}\right)\in\Pi_{r_{1}}\left(\V_{r_{2}}\right)$
with $C_{j}\left(\delta\left(\pi_{r_{2}},\pi_{r_{1}}\right)\right)=\left\{ k\in\left\{ 1,2,\ldots,r_{2}\right\} |C_{k}\left(\pi_{r_{2}}\right)\subseteq C_{j}\left(\pi_{r_{1}}\right)\right\} $.
We can now define the coarsening sequence.

\begin{Definition}[coarsening sequence]\label{ coarsing sequence}Consider
a vertex set $\V$ and its $N$\emph{-chain }$\chi(\V)\subseteq\Pi(\V)$.
Then we define the \textit{coarsening sequence} as $\Delta(\chi)=\left\{ \delta_{i}\right\} _{i=1}^{N-1}$
with $\delta_{i}\triangleq\delta(\pi_{r_{i+1}},\pi_{r_{i}})$.\end{Definition}

The refinement relation is reflexive, anti-symmetric and transitive,
therefore, the set of partitions together with the refinement relation,
$\left(\Pi\left(\V\right),\leq\right)$, falls under the definition
of a finite \textit{partial-ordered set} (poset). Let $\graph$, we
define the contraction set $\G\sslash\Pi\triangleq\left\{ \gcont|\pi\in\Pi\left(\V\right)\right\} $,
and define the contraction binary relation $\gcont_{r_{1}}\leq\gcont_{r_{2}}$
if $\pi_{r_{1}}\leq\pi_{r_{2}}$. Since there is a one-to-one correspondence
between $\left(\G\sslash\Pi,\leq\right)$ and $\left(\Pi\left(\V\right),\leq\right)$,
the contraction set with the contraction binary relation, $\left(\G\sslash\Pi,\leq\right)$,
is also a poset, and for each $N$-chain $\chi\subseteq\Pi\left(\V\right)$
there is a corresponding \emph{contraction chain }$\G\sslash\chi=\left\{ \G\sslash\pi_{r_{i}}\right\} _{i=1}^{N}\subseteq\G\sslash\Pi$.

For each coarsening sequence $\Delta\left(\chi\right)$ we can then
define a corresponding contraction sequence, a series of graphs where
each graph in the series is a graph contraction of the former graph
over the coarsening partition in the coarsening sequence.

\begin{Definition}[contraction sequence]\label{contraction sequence}
Consider a graph $\G$ and an N-chain $\chi(\V)\subseteq\Pi\left(\V\left(\G\right)\right)$
with coarsening sequence $\Delta\left(\chi\right)=\left\{ \delta_{i}\right\} _{i=1}^{N-1}$.
Then we define the \emph{contraction sequence} $\G\sslash\Delta\left(\chi\right)\triangleq\left\{ \G_{i}\right\} _{i=0}^{N-1}$
with $\G_{i}=\G_{i-1}\sslash\delta_{N-i}$ and $\G_{0}=\G\sslash\pi_{r_{N}}$.
\end{Definition}

\begin{prop}Consider a graph $\G$ and its partition $\pi\in\Pi\left(\G\right)$,
and let $\chi=\left\{ \pi_{r_{i}}\right\} _{i=1}^{N}\subseteq\Pi\left(\V\right)$
be a chain with $\pi_{r_{1}}=\pi$ and corresponding contraction sequence
$\G\sslash\Delta\left(\chi\right)=\left\{ \G_{i}\right\} _{i=0}^{N-1}$.
Then $\G_{N-1}=\gcont$.\end{prop}

\begin{proof}It is sufficient to prove for any two-chain $\pi=\pi_{r_{1}}<\pi_{r_{2}}$
with $\Delta\left(\chi\right)=\delta\left(\pi_{r_{2}},\pi_{r_{1}}\right)$,
i.e., $\gcont=\left(\G\sslash\pi_{r_{2}}\right)\sslash\delta\left(\pi_{r_{2}},\pi_{r_{1}}\right)$,
and extend by induction for $N>2$. The order of $\G_{0}=\G\sslash\pi_{r_{2}}$
is $r_{2}$ and from the coarsening sequence (Definition \ref{ coarsing sequence})
we get that the order of $\G_{1}=\left(\G\sslash\pi_{r_{2}}\right)\sslash\delta\left(\pi_{r_{2}},\pi_{r_{1}}\right)$
is $r_{1}=\left|\pi\right|$, therefore, $\V\left(\G_{1}\right)=\V\left(\gcont\right)$.
It is left to show that $\E\left(\G_{1}\right)=\E\left(\gcont\right)$.
Let $\tilde{\epsilon}\in\E\left(\gcont\right)$ then $\exists\epsilon\in\E_{\G}$
such that $\tilde{\epsilon}=f_{\pi}\left(\epsilon\right)$. Now let
$\epsilon_{1}=f_{\pi_{r_{2}}}\left(\epsilon\right)$ and $\epsilon_{2}=f_{\delta}\left(\epsilon_{1}\right)$,
from the coarsening sequence (Definition \ref{ coarsing sequence})
we then obtain that the end nodes of $\epsilon_{2}$ are the end nodes
of $\tilde{\epsilon}$, therefore, $\E\left(\G_{1}\right)=\E\left(\gcont\right)$.
\end{proof}

\begin{coro}[atom-contraction sequence]\label{atom-contraction sequence}Consider
a graph $\G$ and its partition $\pi\in\Pi_{r}\left(\G\right)$ for
$r<n$. Then there exists a chain $\chi\left(\V\right)=\left\{ \pi_{r_{i}}\right\} _{i=1}^{n-r+1}\subseteq\Pi\left(\V_{n}\right)$
such that $\G\sslash\Delta\left(\chi\right)=\left\{ \G_{i}\right\} _{i=0}^{n-r}$
is an atom contraction sequence, i.e., $\delta\left(\pi_{r_{i+1}},\pi_{r_{i}}\right)$
is an atom-partition.\end{coro}

\begin{proof}Choose $\pi_{r_{1}}=\pi\left(\V_{n}\right)$, and then
construct $\pi_{r_{2}}$ by extracting a singleton from a non-singleton
cell of $\pi$. Continue to extract singleton cells until all cells
are singletons, i.e., $\pi_{r_{N}}=\pi_{n}\left(\V_{n}\right)$. The
number of singleton extractions of non-singleton cells in an $r$-partition
is $n-r$, therefore, $N=n-r+1$.\end{proof}

For example, consider the 2-chain $\chi\left(\V_{5}\right)=\left\{ \pi_{2},\pi_{3}\right\} $
with 
\[
\pi_{2}\left(\V_{5}\right)=\left\{ \underbrace{\left\{ v_{1},v_{2},v_{3}\right\} }_{C_{1}},\underbrace{\left\{ v_{4},v_{5}\right\} }_{C_{2}}\right\} ,\mbox{ and }\pi_{3}\left(\V_{5}\right)=\left\{ \underbrace{\left\{ v_{1},v_{2}\right\} }_{C_{1}},\underbrace{\left\{ v_{3}\right\} }_{C_{2}},\underbrace{\left\{ v_{4},v_{5}\right\} }_{C_{3}}\right\} .
\]
We have $C_{1}\left(\pi_{3}\right),C_{2}\left(\pi_{3}\right)\subseteq C_{1}\left(\pi_{2}\right)$
and $C_{3}\left(\pi_{3}\right)\subseteq C_{2}\left(\pi_{2}\right)$,
therefore, $\pi_{3}<\pi_{2}$. We can then construct the \textit{coarsening
sequence }$\Delta\left(\chi\right)=\delta\left(\pi_{3},\pi_{2}\right)$\emph{
}with $\delta\left(\pi_{3},\pi_{2}\right)=\{\underbrace{\left\{ 1,2\right\} }_{C_{1}},\underbrace{\left\{ 3\right\} }_{C_{2}}\}$.
The resulting graph contraction sequence is presented in Figure \ref{fig:Contraction-sequence}.
\begin{center}
\begin{figure}
\centering{}\includegraphics[width=4in]{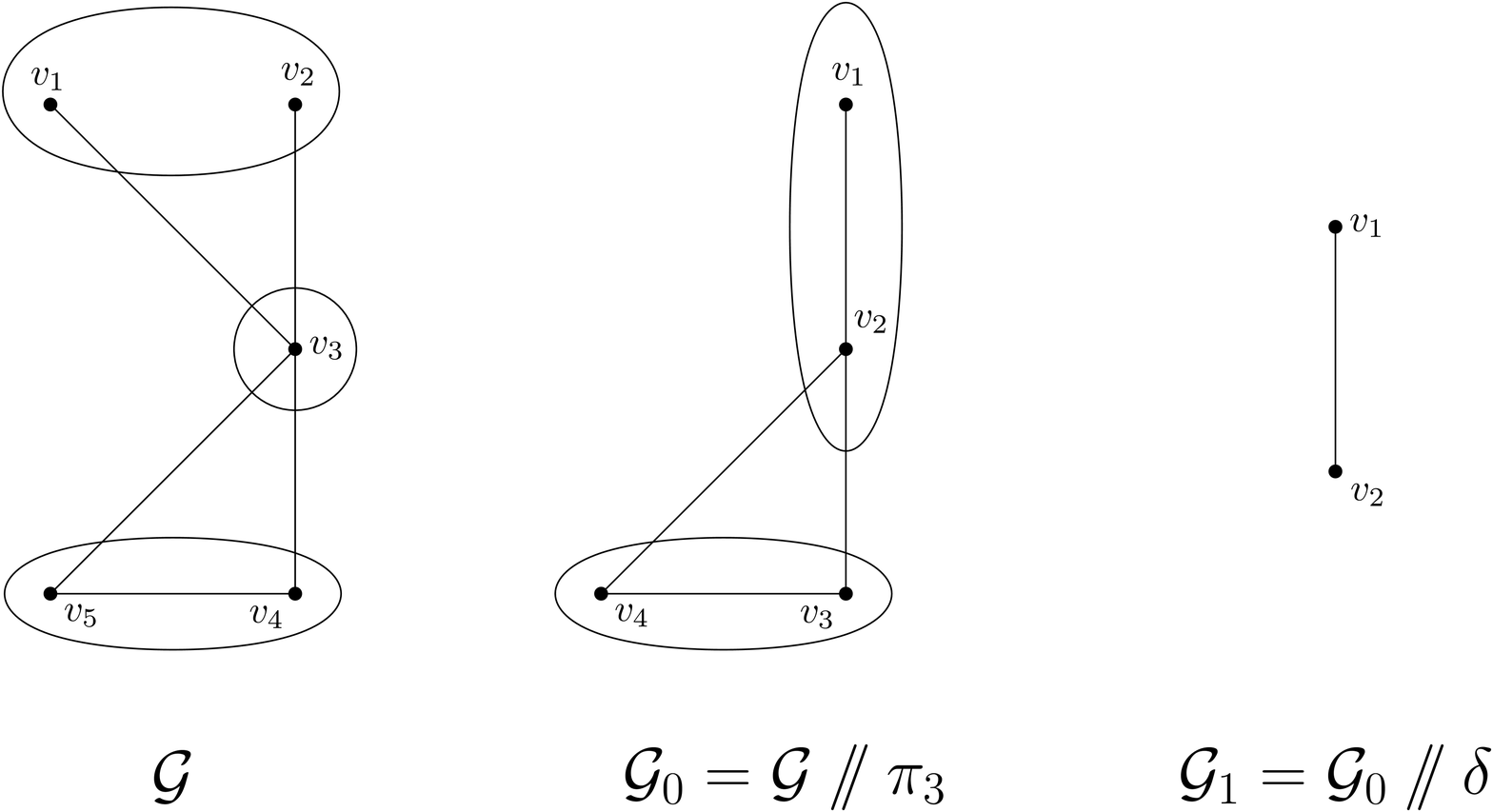}
\caption{The graph contraction sequence leading to $\G\sslash\pi_{2}$: first the contraction $\G_{0}=\G\sslash\pi_{3}$ is performed followed by the contraction over the coarsing $\G_{1}=\G_{0}\sslash\delta\left(\pi_{3},\pi_{2}\right)$. \label{fig:Contraction-sequence} }
\end{figure}
\par\end{center}

\subsection{Edge Contractions}

Graph contractions are defined over vertex partitions. However, there
is also an edge-based approach to perform graph contractions.

\begin{Definition}[edge contraction partition] \label{edge contraction partition}
Consider a graph $\G$ and an \textit{edge contraction set} $\mathcal{E}_{cs}\subset\mathcal{E}(\mathcal{G})$
with $|\mathcal{E}_{cs}|=n-r$ \textit{. }Then we define the \emph{edge
contraction partition} $\pi_{c}\left(\G,\E_{cs}\right)$ as the connected
components partition of the graph $\G_{c}\left(\G,\E_{cs}\right)=\left(\V\left(\G\right),\E_{cs}\right)$,
i.e., $\pi_{c}\left(\G,\E_{cs}\right)=\pi_{cc}\left(\G_{c}\left(\G,\E_{cs}\right)\right)$.
The set of all \textit{edge contraction sets} of cardinality $p$
is defined as $\Xi_{p}\left(\G\right)\triangleq\left\{ \E_{cs}\subset\E\left(\G\right)|\left|\E_{cs}\right|=p\right\} $.
\end{Definition}

With the {edge contraction partition} definition we can define an
edge-based graph contraction.

\begin{Definition}[edge-based graph contraction]\label{edge contraction}
Consider a graph $\G$ and an edge contraction set $\E_{cs}\in\Xi_{n-r}\left(\G\right)$
for $r<n$. Then the \textit{edge-based contraction} is defined as
the contraction over the \emph{edge contraction partition}, i.e.,
$\G\sslash\E_{cs}=\G\sslash\pi_{c}\left(\G,\E_{cs}\right)$.\end{Definition}

In this work we find that a class of edge-matching contractions has
interlacing properties.

\begin{Definition}[edge-matching contraction]\label{edge-matching contraction}
Consider a graph $\G$ and an edge contraction set $\E_{cs}\in\Xi_{n-r}\left(\G\right)$
for $r<n$. Then $\G\sslash\E_{cs}$ is an \textit{edge-matching contraction}
if there is one-to-one correspondence between $\E\left(\G\right)\backslash\E_{cs}$
and $\E\left(\G\sslash\E_{cs}\right)$.\end{Definition} A graph contraction
cannot create new edges, therefore, edge-matching (Definition \ref{edge-matching contraction})
is equivalent to $\left|\E\left(\G\right)\backslash\E_{cs}\right|=\left|\E\left(\G\sslash\E_{cs}\right)\right|$.

\begin{prop}\label{edge-matching cycle invariant}Consider a graph
$\G$ and an edge contraction set $\E_{cs}\in\Xi_{n-r}\left(\G\right)$.
Then if $\G\sslash\E_{cs}$ is cycle-invariant (Definition \ref{cycle-invariat contraction})
it is also edge-matching (Definition \ref{edge-matching contraction}).\end{prop}

\begin{proof} If $\G\sslash\E_{cs}$ is cycle-invariant then from
Definition \ref{cycle-invariat contraction} the edges in $\E_{cs}$
are not part of any cycle of $\G$. Therefore, the contraction does
not map any two edges in $\E\left(\G\right)\backslash\E_{cs}$ to
a single edge in $\E\left(\G\sslash\E_{cs}\right)$, otherwise they
would have been part of a cycle with an edge in $\E_{cs}$, and we
obtain that $\left|\E\left(\G\right)\backslash\E_{cs}\right|=\left|\E\left(\G\sslash\E_{cs}\right)\right|$.\end{proof}

\begin{prop}\label{edge-matching node-removal}Consider a graph $\G$
and a node $v\in\V\left(G\right)$, and let $\pi_{cc}\left(\G\backslash v\right)$
be the connected component partition of $\G\backslash v$, then for
$C_{i}\in\pi_{cc}\left(\G\backslash v\right)$ and $\E_{cs}=\E\left(\G\left[C_{i}\cup v\right]\right)$,
the contraction $\G\sslash\E_{cs}$ is node-removal equivalent (Definition
\ref{node-removal equivalent contraction}) with $\V_{S}=C_{i}$,
and is also edge-matching (Definition \ref{edge-matching contraction}).\end{prop}

\begin{proof} Since $C_{i}$ is a connected component of $\G\backslash\E\left(\G\left[\N_{v}\cup v\right]\right)$
then $v$ is the only node in any path between $C_{i}$ and $\V\left(\G\right)\backslash\left\{ C_{i}\cup v\right\} $,
therefore, by choosing $\V_{S}=C_{i}$ the graph $\G\backslash C_{i}$
removes all edges $\E\left(\G\left[C_{i}\right]\right)$ and all edges
connecting $C_{i}$ to $\V\left(\G\right)\backslash C_{i}$ which
are the edges between $C_{i}$ and $v$ and we obtain that $\G\backslash C_{i}=\G\sslash\E\left(\G\left[C_{i}\cup v\right]\right)$,
i.e., the contraction $\G\sslash\E_{cs}$ is node-removal equivalent
(Definition \ref{node-removal equivalent contraction}). Furthermore,
contracting all edges $\E\left(\G\left[C_{i}\cup v\right]\right)$
does not effect any other edges in $\G$ such that $\left|\E\left(\G\right)\backslash\E_{cs}\right|=\left|\E\left(\G\sslash\E_{cs}\right)\right|$
and we obtain that $\G\sslash\E_{cs}$ is edge-matching.\end{proof}

We can choose a subset of tree edges to create a tree-based contraction
of a graph.

\begin{Definition}[tree-based contraction] \label{tree-based contraction}Consider
a graph $\G$ and its spanning tree $\T\in\Tree$ with an edge contraction
set $\E_{cs}\in\Xi_{n-r}\left(\T\right)$. Then $\G\sslash\E_{cs}$
is a \textit{tree-based contraction}. \end{Definition}

For example, the graph contraction $\gcont$ presented in Figure \ref{fig:Cycle-invariant-contraction}
can also be performed as an edge-based contraction $\G\sslash\E_{cs}$
with $\E_{cs}=\left\{ \left\{ v_{1},v_{3}\right\} ,\left\{ v_{2},v_{3}\right\} \right\} $
and a tree-based contraction (Definition \ref{tree-based contraction}).

If the contraction edge set is a subset of the edges of a spanning
tree, then the contracted tree edges will form a spanning tree of
the contracted graph.

\begin{prop}\label{tree2tree edge-based contraction}Consider a graph
$\G$ and its spanning tree $\T\in\Tree$ with an edge contraction
set $\E_{cs}\in\Xi_{n-r}\left(\T\right)$. Then $\T\sslash\E_{cs}\in\mathbb{T}\left(\G\sslash\E_{cs}\right)$,
i.e., $\T\sslash\E_{cs}$ is a tree of order $r$ of the contracted
graph.\end{prop}

\begin{proof} A tree of order $n$ has $n-1$ edges, and by contracting
$n-r$ tree edges we are left with $\left(n-1\right)-\left(n-r\right)$
edges, such that $\left|\E\left(\T\sslash\E_{cs}\right)\right|=r-1$.
It is left to show that $\T\sslash\E_{cs}\left(\T\right)\subseteq\G\sslash\E_{cs}\left(\T\right)$.
From Lemma \ref{If--is connected} we obtain that $\T\sslash\E_{cs}$
is connected, therefore, $\T\sslash\E_{cs}$ is a connected graph
of order $r$ with $r-1$ edges, which is a tree of order $r$. Since
$\E_{cs}\left(\T\right)\subseteq\E\left(\G\right)$ we have $\pi_{c}\left(\T,\E_{cs}\left(\T\right)\right)=\pi_{c}\left(\G,\E_{cs}\left(\T\right)\right)$,
and since $\T=\G\backslash\E\left(\C\right)$ we obtain from the subgraph
contraction lemma (Lemma \ref{subgraph contraction}) that $\T\sslash\pi_{c}\left(\T,\E_{cs}\left(\T\right)\right)\subseteq\G\sslash\pi_{c}\left(\T,\E_{cs}\left(\T\right)\right)$
and conclude that $\T\sslash\E_{cs}\left(\T\right)\subseteq\G\sslash\E_{cs}\left(\T\right)$,
and therefore, $\T\sslash\E_{cs}\left(\T\right)\in\mathbb{T}\left(\G\sslash\E_{cs}\right)$.\end{proof}

\begin{prop}\label{degree contraction bound} Consider a graph $\G$
and an edge contraction set $\E_{cs}\in\Xi_{n-r}\left(\G\right)$.
Then $\forall\tilde{v}\in\V\left(\G\sslash\E_{cs}\right)$ 
\begin{equation}
d_{\tilde{v}}\left(\G\sslash\E_{cs}\right)\leq\left(\sum\limits _{v\in C_{\tilde{v}}\left(\pi\right)}d_{v}\left(\G\right)\right)-2\left(\left|C_{\tilde{v}}\left(\pi\right)\right|-1\right),
\end{equation}
where $\pi=\pi_{c}\left(\G,\E_{cs}\right)$.\end{prop}

\begin{proof} From Proposition \ref{degree-contraction} we obtain
that $d_{\tilde{v}}\left(\gcont\right)=\left|f_{\pi}\left(\N_{C_{\tilde{v}}}\right)\right|$.
We have $\left|f_{\pi}\left(\N_{C_{\tilde{v}}}\right)\right|\leq\left|\N_{C_{\tilde{v}}}\right|$
and since $C_{\tilde{v}}\left(\pi\right)\in\pi_{c}$ is a connected
component of $\G$ we get 
\begin{align*}
\left|\N_{C_{\tilde{v}}}\right| & \leq\left(\sum\limits _{v\in C_{\tilde{v}}\left(\pi\right)}d_{v}\left(\G\right)\right)-2\left|\E\left(\G\left[C_{\tilde{v}}\left(\pi\right)\right]\right)\right|.
\end{align*}

The number of edges in the cell $\left|\E\left(\G\left[C_{\tilde{v}}\left(\pi\right)\right]\right)\right|$
is at least the number of spanning tree edges, therefore, $\left|\E\left(\G\left[C_{\tilde{v}}\left(\pi\right)\right]\right)\right|\geq\left|C_{\tilde{v}}\left(\pi\right)\right|-1$,
and we obtain that 
\[
d_{\tilde{v}}\left(\G\sslash\E_{cs}\right)\leq\left(\sum\limits _{v\in C_{\tilde{v}}\left(\pi\right)}d_{v}\left(\G\right)\right)-2\left(\left|C_{\tilde{v}}\left(\pi\right)\right|-1\right),
\]
completing the proof.\end{proof}

\begin{coro}\label{cycle-invariant degree contraction} Consider
a graph $\G$ and an edge contraction set $\E_{cs}\in\Xi_{n-r}\left(\G\right)$
for $r<n$. Then if $\G\sslash\E_{cs}$ is cycle-invariant (Definition
\ref{node-removal equivalent contraction}) then $\forall\tilde{v}\in\V\left(\G\sslash\E_{cs}\right)$,
\begin{equation}
d_{\tilde{v}}\left(\G\sslash\E_{cs}\right)=\left(\sum\limits _{v\in C_{\tilde{v}}\left(\pi\right)}d_{v}\left(\G\right)\right)-2\left(\left|C_{\tilde{v}}\left(\pi\right)\right|-1\right),
\end{equation}
where $\pi=\pi_{c}\left(\G,\E_{cs}\right)$.\end{coro}

\begin{proof} Since $\forall\tilde{v}\in\V\left(\G\sslash\E_{cs}\right)$
$C_{\tilde{v}}\left(\pi\right)$ is a connected component of $\G$,
and $\G\sslash\E_{cs}$ is cycle-invariant then $\left|f_{\pi}\left(\N_{C_{\tilde{v}}}\right)\right|=\left|\N_{C_{\tilde{v}}}\right|$
and $\G\left[C_{\tilde{v}}\left(\pi\right)\right]$ is a tree of order
$\left|C_{\tilde{v}}\left(\pi\right)\right|$, such that from Proposition
\ref{degree-contraction} we obtain that 
\[
d_{\tilde{v}}\left(\G\sslash\E_{cs}\right)=\left(\sum\limits _{v\in C_{\tilde{v}}\left(\pi\right)}d_{v}\left(\G\right)\right)-2\left(\left|C_{\tilde{v}}\left(\pi\right)\right|-1\right).
\]
\end{proof}

\begin{coro}\label{tree is edge-matching}If a graph $\G$ is a tree
then $\G\sslash\E_{cs}$ is edge-matching for any $\E_{cs}\in\Xi_{n-r}\left(\G\right)$.\end{coro}
\begin{proof} If $\G$ is a tree then $\G\sslash\E_{cs}$ is cycle-invariant
for any $\E_{cs}\in\Xi_{n-r}\left(\G\right)$ and from Proposition
\ref{edge-matching cycle invariant} we obtain that $\G\sslash\E_{cs}$
is edge-matching.\end{proof}

Trees and cycle-completing edges are the building blocks of any connected
graph, and this tree and co-tree structure is described by the Tucker
representation \citep[p.113]{rockafellar1984network}.

\begin{Definition}[Tucker representation]\label{Tucker_representation}
Consider a graph $\G$ and its spanning tree $\T\in\Tree$ with co-tree
$\C\left(\T\right)$, with arbitrary head and tail assigned to the
end-nodes of each edge in $\E\left(\G\right)$. For each edge $\epsilon_{j}\in\E\left(\C\right)$
there is a path from head to tail in $\T$, and we define a corresponding
\textit{signed path vector} $t_{j}\in\R^{\left|\E\left(\T\right)\right|}$,
$\left[t_{j}\right]_{k}=1$ if $\epsilon_{k}\left(\T\right)$ (with
the assigned head and tail) is along the path, $\left[t_{j}\right]_{k}=-1$
if $\epsilon_{k}\left(\T\right)$ is opposite to the path, and $\left[t_{j}\right]_{k}=0$
otherwise. The \emph{Tucker representation} of the co-tree is then
the matrix $T_{\left(\T,\mathcal{C}\right)}\in\R^{\left|\E\left(\T\right)\right|\times\left|\E\left(\C\right)\right|}$
where the $j$th column of $T_{\left(\T,\mathcal{C}\right)}$ is the
signed path vector $t_{j}\in\R^{\left|\E\left(\T\right)\right|}$
of the corresponding edge $\epsilon_{j}\in\E\left(\C\right)$. \end{Definition}

\begin{prop}\label{find cycle invariant contraction}Consider a graph
$\G$ and an edge contraction set $\E_{cs}\in\Xi_{n-r}\left(\G\right)$
for $r<n$, and let $\T\in\mathbb{T}\left(\G\right)$. Then $\G\sslash\E_{cs}$
is cycle-invariant (Definition \ref{cycle-invariat contraction})
if and only if $\E_{cs}\subseteq\E\left(\T\right)$ and the corresponding
rows of $T_{\left(\T,\mathcal{C}\right)}$ are all zeros.\end{prop}
\begin{proof}If $\G\sslash\E_{cs}$ is cycle-invariant then from
Definition \ref{cycle-invariat contraction} the edges in $\E_{cs}$
are not part of any cycle of $\G$, therefore, $\E_{cs}\subseteq\E\left(\T\right)$
for any $\T\in\mathbb{T}\left(\G\right)$. If $\epsilon\in\E\left(\T\right)$
is not part of any cycle in $\G$ then form the Tucker representation
(Definition \ref{Tucker_representation}) we get that the corresponding
row of $T_{\left(\T,\mathcal{C}\right)}$ is all zeros.

If $\E_{cs}\subseteq\E\left(\T\right)$ and the corresponding rows
of $T_{\left(\T,\mathcal{C}\right)}$ are all zeros, then the edges
in $\E_{cs}$ are not part of any cycle in $\G$, such that the tree-based
contraction (Definition \ref{tree-based contraction}) $\G\sslash\E_{cs}$
is cycle-invariant.\end{proof}

\section{Interlacing Graph Contractions\label{sec:Interlacing-Graph-Contractions}}

The general interlacing graph reduction problem (Problem \ref{interlacing graph reduction})
is combinatorial hard. If we restrict the class of reduced-order graphs
to graph contractions then we get the following interlacing graph
contraction problem.

\begin{prob}[interlacing graph contraction]\label{interlacing graph contraction}Consider
a graph $\G$ and a real symmetric graph matrix $M\left(\G\right)\in\R^{n\times n}$.
Then given $r<n$ find $\pi\in\Pi_{r}\left(\G\right)$ such that $\gcont\propto_{M}\G$.\end{prob}

The number of $r$-\textit{partitions} is $\left|\Pi_{r}\left(\G\right)\right|=S\left(n,r\right)$
where 
\begin{equation*}
S\left(n,r\right)=\sum\limits _{k=1}^{r}\left(-1\right)^{r-k}\frac{k^{n}}{k!\left(r-k\right)!},
\end{equation*}
is the Stirling number of the second kind \citep[p.18]{wilf1994generatingfunctionology},
which for $r\ll n$ is asymptotically $S\left(n,r\right)\sim\frac{r^{n}}{r!}$.
If we restrict the problem to edge-based contractions then the number
of partitions is the number of \textit{$n-r$ edge contractions} is
$\left|\Xi_{n-r}\left(\G\right)\right|=\left(\begin{array}{c}
m\\
n-r
\end{array}\right)$ where $m=\left|\E\left(\G\right)\right|$. Finding an interlacing
contraction is, therefore, combinatorial hard as well and in the following
section we show how cycle-invariant and node-removal equivalent contractions
have associated subspaces required by Theorem \ref{interlacing graph reduction theorem}
and lead to interlacing graphs.

Consider a graph $\graph$ of order $n$ and an $r$-partition $\pi\in\Pi_{r}\left(\G\right)$
and consider a subset $\V_{S}\subset\V\left(\G\right)$, $\left|\V_{S}\right|=n-r$
for $r<n$. Then we define the following subspaces of dimension $r$.
The\textit{ partition subspace} $\mathcal{F}_{\pi}\subseteq\R^{n}$
is the space of all vectors in $\R^{n}$ such that variables with
indexes in the same partition cell are equal, 
\begin{equation}
\mathcal{F}_{\pi}\triangleq\left\{ x\in\R^{n}|x_{j}=x_{k},\forall j,k\in C_{i}\left(\pi\right),\forall i\in\left[1,r\right]\right\} ,\label{eq:partition space}
\end{equation}
and the corresponding \textit{partition mapping} $p_{\mathcal{F}_{\pi}}\left(\tilde{x}\right):\R^{r}\rightarrow\R^{n}$,
 
\begin{equation}
\left[p_{\mathcal{F}_{\pi}}\left(\tilde{x}\right)\right]_{k}=\left\{ \tilde{x}_{i}|k\in C_{i}\left(\pi\right)\right\} .\label{eq:partition lifting}
\end{equation}

We define the \textit{anti-partition subspace} $\tilde{\mathcal{F}}_{\pi}\subseteq\R^{n}$
such that for $x\in\tilde{\mathcal{F}}_{\pi}$ the sum of all vector
variables in non-singleton partition cells is zero 
\begin{equation}
\tilde{\mathcal{F}}_{\pi}\triangleq\left\{ x\in\R^{n}|x_{v_{j}\left(C_{i}\left(\pi\right)\right)}=-\frac{x_{v_{1}\left(C_{i}\left(\pi\right)\right)}}{\left|C_{i}\left(\pi\right)\right|-1},\begin{array}{c}
\forall i\in\left[1,r\right],\text{\ensuremath{\left|C_{i}\left(\pi\right)\right|>1}}\\
\forall j\in\left[2,\left|C_{i}\left(\pi\right)\right|\right]
\end{array}\right\} ,\label{eq:anti-partition subspace}
\end{equation}
 and the corresponding \textit{anti-partition mapping}, $p_{\tilde{\mathcal{F}}_{\pi}}\left(\tilde{x}\right):\R^{r}\rightarrow\R^{n}$,
\begin{equation}
\left[p_{\tilde{\mathcal{F}}_{\pi}}\left(\tilde{x}\right)\right]_{k}=\begin{cases}
\tilde{x}_{k} & k=v_{1}\left(C_{i}\left(\pi\right)\right)\\
-\frac{\tilde{x}_{k}}{\left|C_{i}\left(\pi\right)\right|-1} & k=v_{j}\left(C_{i}\left(\pi\right)\right),j\geq2
\end{cases},\label{eq:anti-partition lifting}
\end{equation}
where $v_{j}\left(C_{i}\left(\pi\right)\right)$ denotes the j'th
node of the i'th partition cell.

The \textit{node-removal subspace}, $\mathcal{F}_{\V_{S}}\subseteq\R^{n}$,
is defined as 
\begin{equation}
\mathcal{F}_{\V_{S}}\triangleq\left\{ x\in\R^{n}|x_{i}=0,i\in\V_{S}\right\} ,\label{eq:node-removal space}
\end{equation}
and the corresponding \textit{node-removal mapping} $p_{\mathcal{F}_{\V_{S}}}\left(\tilde{x}\right):\R^{r}\rightarrow\R^{n}$,
\begin{equation}
\left[p_{\mathcal{F}_{\V_{S}}}\left(\tilde{x}\right)\right]_{k}=\begin{cases}
\tilde{x}_{k} & k\notin\V_{S}\\
0 & o.w.
\end{cases}.\label{eq:node-removal lifting}
\end{equation}

\begin{prop}\label{Rayleigh quotient edge-matching and node-removal equivalent contraction  inequalities}Consider
a graph $\G$ and an edge-matching and node-removal equivalent contraction
$\G\sslash\E_{cs}$ (Definitions \ref{edge-matching contraction}\&\ref{edge contraction partition})
with $\E_{cs}\in\Xi_{n-r}\left(\G\right)$ for $r<n$. Then for $\tilde{x}\in\R^{r}$
we have 
\begin{equation*}
R\left(L\left(\G\right),p_{\mathcal{F}_{\pi}}\left(\tilde{x}\right)\right)\leq R\left(L\left(\G\sslash\E_{cs}\right),\tilde{x}\right),
\end{equation*}
and 
\begin{equation*}
R\left(L\left(\G\right),p_{\mathcal{F}_{\V_{S}}}\left(\tilde{x}\right)\right)\geq R\left(L\left(\G\sslash\E_{cs}\right),\tilde{x}\right).
\end{equation*}

\end{prop}

\begin{proof}Let $x=p_{\mathcal{F}_{\pi}}\left(\tilde{x}\right)$
for $\tilde{x}\in\R^{r}$. The Rayleigh quotients of the Laplacian
takes the form \cite{chen2004interlacing} 
\begin{equation*}
R\left(L\left(\G\right),x\right)=\frac{\sum\limits _{\left\{ u,v\right\} \in\E\left(\G\right)}\left(x_{v}-x_{u}\right)^{2}}{\sum\limits _{v\in\V\left(\G\right)}x_{v}^{2}}.\label{eq:Rayleigh quotient of the Laplacian-1}
\end{equation*}

Separating the edges to $\E_{cs}$ and $\E\backslash\E_{cs}$, the
sum $\sum\limits _{\left\{ u,v\right\} \in\E\left(\G\right)}\left(x_{v}-x_{u}\right)^{2}$
can be written as 
\begin{equation*}
\sum\limits _{\left\{ u,v\right\} \in\E\left(\G\right)}\left(x_{v}-x_{u}\right)^{2}=\sum\limits _{\left\{ u,v\right\} \in\E\left(\G\right)\backslash\E_{cs}}\left(x_{u}-x_{v}\right)^{2}+\sum\limits _{\left\{ u,v\right\} \in\E_{cs}}\left(x_{u}-x_{v}\right)^{2}.\label{eq:Seperating the edges-1}
\end{equation*}

Therefore, if $x\in\mathcal{F}_{\pi}$ and $\left\{ u,v\right\} \in\E_{cs}$
then $\sum\limits _{\left\{ u,v\right\} \in\E_{cs}}\left(x_{u}-x_{v}\right)^{2}=0$
and 
\begin{equation*}
\sum\limits _{\left\{ u,v\right\} \in\E}\left(x_{v}-x_{u}\right)^{2}=\sum\limits _{\left\{ u,v\right\} \in\E\backslash\E_{cs}}\left(x_{u}-x_{v}\right)^{2}.
\end{equation*}

Since $\G\sslash\E_{cs}$ is edge-matching (Definition \ref{edge-matching contraction})
there is one-to-one correspondence between $\E\left(\G\right)\backslash\E_{cs}$
and $\E\left(\G\sslash\E_{cs}\right)$ (Proposition \ref{edge-matching cycle invariant}),
and substituting the partition lifting $x=p_{\mathcal{F}_{\pi}}\left(\tilde{x}\right)$
(Eq. \eqref{eq:partition lifting}) we get 
\begin{equation*}
\sum\limits _{\left\{ u,v\right\} \in\E\backslash\E_{cs}}\left(x_{u}-x_{v}\right)^{2}=\sum\limits _{\left\{ u,v\right\} \in\E\left(\G\sslash\E_{cs}\right)}\left(\tilde{x}_{u}-\tilde{x}_{v}\right)^{2}.\label{eq:edge-sum equality-1}
\end{equation*}

Rearranging the sums $\sum\limits _{v\in\V}x_{v}^{2}$ over the vertices
of each partition cell and substituting the partition lifting $x=p_{\mathcal{F}_{\pi}}\left(\tilde{x}\right)$
(Eq. \eqref{eq:partition lifting}) we get, 
\begin{align*}
\sum\limits _{v\in\V\left(\G\right)}x_{v}^{2} & =\sum_{i=1}^{r}\sum\limits _{v\in C_{i}\left(\pi\right)}x_{v}^{2}\nonumber \\
 & =\sum\limits _{u\in\V\left(\G\sslash\E_{cs}\right)}\tilde{x}_{u}^{2}\left|C_{u}\left(\pi\right)\right|,
\end{align*}

The Rayleigh quotients of the Laplacian is then 
\begin{equation*}
R\left(L\left(\G\right),p_{\mathcal{F}_{\pi}}\left(\tilde{x}\right)\right)=\frac{\sum\limits _{\left\{ u,v\right\} \in\E\left(\G\sslash\E_{cs}\right)}\left(\tilde{x}_{u}-\tilde{x}_{v}\right)^{2}}{\sum\limits _{u\in\V\left(\G\sslash\E_{cs}\right)}\tilde{x}_{u}^{2}\left|C_{u}\left(\pi\right)\right|},
\end{equation*}
and we have $\left|C_{i}\left(\pi\right)\right|\geq1$, therefore,
\begin{align*}
R\left(L\left(\G\right),p_{\mathcal{F}_{\pi}}\left(\tilde{x}\right)\right) & \leq\frac{\sum\limits _{\left\{ u,v\right\} \in\E\left(\G\sslash\E_{cs}\right)}\left(\tilde{x}_{u}-\tilde{x}_{v}\right)^{2}}{\sum\limits _{u\in\V\left(\G\sslash\E_{cs}\right)}\tilde{x}_{u}^{2}}\nonumber \\
 & =R\left(L\left(\G\sslash\E_{cs}\right),\tilde{x}\right).
\end{align*}

If $\G\sslash\E_{cs}$ is node-removal equivalent (Definition \ref{node-removal equivalent contraction})
then by substituting the node-removal lifting $x=p_{\mathcal{F}_{\V_{S}}}\left(\tilde{x}\right)$
(Eq. \eqref{eq:partition lifting}) we get 
\[
\sum\limits _{\left\{ u,v\right\} \in\E\backslash\E_{cs}}\left(x_{u}-x_{v}\right)^{2}=\sum\limits _{\left\{ u,v\right\} \in\E\left(\G\sslash\E_{cs}\right)}\left(\tilde{x}_{u}-\tilde{x}_{v}\right)^{2}
\]
and 
\begin{equation*}
\sum\limits _{v\in\V\left(\G\right)}x_{v}^{2}=\sum\limits _{u\in\V\left(\G\sslash\E_{cs}\right)}\tilde{x}_{u}^{2},
\end{equation*}
and we obtain that 
\begin{align*}
R\left(L\left(\G\right),p_{\mathcal{F}_{\V_{S}}}\left(\tilde{x}\right)\right) & \geq\frac{\sum\limits _{\left\{ u,v\right\} \in\E\backslash\E_{cs}}\left(x_{v}-x_{u}\right)^{2}}{\sum\limits _{v\in\V\left(\G\right)}x_{v}^{2}}\nonumber \\
 & =\frac{\sum\limits _{\left\{ u,v\right\} \in\E\left(\G\sslash\E_{cs}\right)}\left(\tilde{x}_{u}-\tilde{x}_{v}\right)^{2}}{\sum\limits _{u\in\V\left(\G\sslash\E_{cs}\right)}\tilde{x}_{u}^{2}}\nonumber \\
 & =R\left(L\left(\G\sslash\E_{cs}\right),\tilde{x}\right).
\end{align*}
\end{proof}

\begin{prop}\label{Rayleigh quotient cycle invariant contraction inequalities}Consider
a graph $\G$ and a cycle invariant contraction $\G\sslash\E_{cs}$
(Definition \ref{cycle-invariat contraction}) with $\E_{cs}\in\Xi_{n-r}\left(\G\right)$
for $r<n$. Then for $\tilde{x}\in\R^{r}$ we have 
\begin{equation*}
R\left(\mathcal{L}\left(\G\right),p_{\mathcal{F}_{\pi}}\left(\tilde{x}\right)\right)\leq R\left(\mathcal{L}\left(\G\sslash\E_{cs}\right),\tilde{x}\right),
\end{equation*}
and if $\G\sslash\E_{cs}$ is a single edge contraction with $\E_{cs}=\varepsilon_{cs}$
then 
\begin{equation*}
R\left(\mathcal{L}\left(\G\right),p_{\tilde{\mathcal{F}}_{\pi}}\left(\tilde{x}\right)\right)\geq R\left(\mathcal{L}\left(\G\sslash\varepsilon_{cs}\right),\tilde{x}\right).
\end{equation*}

\end{prop}

\begin{proof}The Rayleigh quotient of the normalized-Laplacian takes
the form \cite{chen2004interlacing} 
\begin{equation}
R\left(\mathcal{L}\left(\G\right),x\right)=\frac{\sum\limits _{\left\{ u,v\right\} \in\E\left(\G\right)}\left(x_{v}-x_{u}\right)^{2}}{\sum\limits _{v\in\V\left(\G\right)}x_{v}^{2}d_{v}\left(\G\right)}.\label{eq:Rayleigh quotient of the normalized-Laplacian}
\end{equation}

Since $\G\sslash\E_{cs}$ is a cycle-invariant contraction it is edge-matching
(Proposition \ref{edge-matching cycle invariant}) and there is one-to-one
correspondence between $\E\left(\G\right)\backslash\E_{cs}$ and $\E\left(\G\sslash\E_{cs}\right)$
(Proposition \ref{edge-matching cycle invariant}), and substituting
the partition mapping $x=p_{\mathcal{F}_{\pi}}\left(\tilde{x}\right)$
for $\tilde{x}\in\R^{r}$ (Eq. \eqref{eq:partition lifting}) we get
as in Eq. \eqref{eq:edge-sum equality-1} 
\begin{align*}
\sum\limits _{\left\{ u,v\right\} \in\E\left(\G\right)}\left(x_{v}-x_{u}\right)^{2} & =\sum\limits _{\left\{ u,v\right\} \in\E\backslash\E_{cs}}\left(x_{u}-x_{v}\right)^{2}\nonumber \\
 & =\sum\limits _{\left\{ u,v\right\} \in\E\left(\G\sslash\E_{cs}\right)}\left(\tilde{x}_{u}-\tilde{x}_{v}\right)^{2}.
\end{align*}

Rearranging the sum $\sum\limits _{v\in\V\left(\G\right)}x_{v}^{2}d_{v}\left(\G\right)$
over the vertices of each partition cell and substituting the partition
lifting $x=p_{\mathcal{F}_{\pi}}\left(\tilde{x}\right)$ (Eq. \eqref{eq:partition lifting})
we get, 
\begin{align*}
\sum\limits _{v\in\V\left(\G\right)}x_{v}^{2}d_{v}\left(\G\right) & =\sum_{i=1}^{r}\sum\limits _{v\in C_{i}\left(\pi\right)}x_{v}^{2}d_{v}\left(\G\right),\nonumber \\
 & =\sum\limits _{u\in\V\left(\G\sslash\E_{cs}\right)}\tilde{x}_{u}^{2}\left(\sum\limits _{v\in C_{u}\left(\pi\right)}d_{v}\left(\G\right)\right).
\end{align*}

The graph contraction $\G\sslash\E_{cs}$ is cycle-invariant, therefore,
from Proposition~\ref{cycle-invariant degree contraction} we have
$d_{u}\left(\G\sslash\E_{cs}\right)=\left(\sum\limits _{v\in C_{u}\left(\pi\right)}d_{v}\left(\G\right)\right)-2\left(\left|C_{u}\left(\pi\right)\right|-1\right)$,
and 
\begin{equation*}
\sum\limits _{v\in\V\left(\G\right)}x_{v}^{2}d_{v}\left(\G\right)=\sum\limits _{u\in\V\left(\G\sslash\E_{cs}\right)}\tilde{x}_{u}^{2}\left[d_{u}\left(\G\sslash\E_{cs}\right)+2\left(\left|C_{u}\left(\pi\right)\right|-1\right)\right].
\end{equation*}

The Rayleigh quotients of the normalized-Laplacian (Eq. \eqref{eq:Rayleigh quotient of the normalized-Laplacian}) is then 
\begin{equation*}
R\left(\mathcal{L}\left(\G\right),p_{\mathcal{F}_{\pi}}\left(\tilde{x}\right)\right)=\frac{\sum\limits _{\left\{ u,v\right\} \in\E\left(\G\sslash\E_{cs}\right)}\left(\tilde{x}_{u}-\tilde{x}_{v}\right)^{2}}{\sum\limits _{u\in\V\left(\G\sslash\E_{cs}\right)}\tilde{x}_{u}^{2}\left[d_{u}\left(\G\sslash\E_{cs}\right)+2\left(\left|C_{u}\left(\pi\right)\right|-1\right)\right]}.
\end{equation*}

We have $\left|C_{i}\left(\pi\right)\right|\geq1$ such that $2\left(\left|C_{i}\left(\pi\right)\right|-1\right)\geq0$,
therefore, 
\begin{align*}
R\left(\mathcal{L}\left(\G\right),p_{\mathcal{F}_{\pi}}\left(\tilde{x}\right)\right) & \leq\frac{\sum\limits _{\left\{ u,v\right\} \in\E\left(\G\sslash\E_{cs}\right)}\left(\tilde{x}_{u}-\tilde{x}_{v}\right)^{2}}{\sum\limits _{u\in\V\left(\G\sslash\E_{cs}\right)}\tilde{x}_{u}^{2}d_{u}\left(\G\sslash\E_{cs}\right)}=R\left(\mathcal{L}\left(\G\sslash\E_{cs}\right),\tilde{x}\right).
\end{align*}

Let $\G\sslash\varepsilon_{cs}$ be a cycle-invariant edge contraction
with corresponding edge contraction partition $\pi\in\Pi_{n-1}\left(\G\right)$.
For an atom-contraction there is only one non-singlet cell, and without
loss of generality we can choose it to be $C_{n-1}\left(\pi\right)=\left\{ n-1,n\right\} $
such that the contracted edge is $\varepsilon_{cs}=\left\{ x_{n-1},x_{n}\right\} $,
and 
\begin{equation*}
R\left(\mathcal{L}\left(\G\right),x\right)=\frac{\sum\limits _{\left\{ u,v\right\} \in\E\backslash\E_{cs}}\left(x_{u}-x_{v}\right)^{2}+\left(x_{n-1}-x_{n}\right)^{2}}{\sum\limits _{v=1}^{n-2}x_{v}^{2}d_{v}\left(\G\right)+x_{n-1}^{2}d_{n-1}\left(\G\right)+x_{n}^{2}d_{n}\left(\G\right)}.
\end{equation*}

\sloppy For this atom-contraction, we have the anti-partition space $\tilde{\mathcal{F}}_{\pi}=\left\{ x\in\mathbb{R}^{n}\,|x_{n-1}=-x_{n}\right\} $
(Eq. \eqref{eq:anti-partition subspace}) and anti-partition mapping
$p_{\tilde{\mathcal{F}}_{\pi}}\left(\tilde{x}\right):\R^{n-1}\rightarrow\R^{n}$
is (Eq. \eqref{eq:anti-partition lifting}) 
\begin{equation*}
\left[p_{\tilde{\mathcal{F}}_{\pi}}\left(\tilde{x}\right)\right]_{k}=\begin{cases}
\tilde{x}_{k} & k\leq n-1\\
-\tilde{x}_{n-1} & k=n
\end{cases},\label{eq:the partition null-space lifting}
\end{equation*}
such that 
\begin{equation*}
R\left(\mathcal{L}\left(\G\right),p_{\tilde{\mathcal{F}}_{\pi}}\left(\tilde{x}\right)\right)=\frac{\sum\limits _{\left\{ u,v\right\} \in\E\backslash\E_{cs}}\left(\tilde{x}_{u}-\tilde{x}_{v}\right)^{2}+4\tilde{x}_{n-1}^{2}}{\sum\limits _{v=1}^{n-2}\tilde{x}_{v}^{2}d_{v}\left(\G\right)+\tilde{x}_{n-1}^{2}\left(d_{n-1}\left(\G\right)+d_{n}\left(\G\right)\right)}.
\end{equation*}

There is one-to-one correspondence between $\E\left(\G\right)\backslash\varepsilon_{cs}$
and $\E\left(\G\sslash\varepsilon_{cs}\right)$ (Proposition \ref{edge-matching cycle invariant}),
therefore, $\sum\limits _{\left\{ u,v\right\} \in\E\backslash\E_{cs}}\left(\tilde{x}_{u}-\tilde{x}_{v}\right)^{2}=\sum\limits _{\left\{ u,v\right\} \in\E\left(\G\sslash\varepsilon_{cs}\right)}\left(\tilde{x}_{u}-\tilde{x}_{v}\right)^{2}$,
and from Proposition \ref{cycle-invariant degree contraction} we
get 
\begin{equation*}
d_{v}\left(\G\sslash\varepsilon_{cs}\right)=\begin{cases}
d_{v}\left(\G\right) & v\leq n-2\\
d_{n-1}\left(\G\right)+d_{n}\left(\G\right)-2 & v=n-1
\end{cases},
\end{equation*}
such that 
\begin{align*}
R\left(\mathcal{L}\left(\G\right),p_{\tilde{\mathcal{F}}_{\pi}}\left(\tilde{x}\right)\right) & =\frac{\sum\limits _{\left\{ u,v\right\} \in\E\left(\G\sslash\varepsilon_{cs}\right)}\left(\tilde{x}_{u}-\tilde{x}_{v}\right)^{2}+4\tilde{x}_{n-1}^{2}}{\sum\limits _{v\in\V\left(\G\sslash\varepsilon_{cs}\right)}\tilde{x}_{v}^{2}d_{v}\left(\G\sslash\varepsilon_{cs}\right)+2\tilde{x}_{n-1}^{2}}\nonumber \\
 & =R\left(\mathcal{L}\left(\G\sslash\varepsilon_{cs}\right),\tilde{x}\right)\frac{1+\frac{4\tilde{x}_{n-1}^{2}}{\sum\limits _{\left\{ u,v\right\} \in\E\left(\G\sslash\varepsilon_{cs}\right)}\left(\tilde{x}_{u}-\tilde{x}_{v}\right)^{2}}}{1+\frac{2\tilde{x}_{n-1}^{2}}{\sum\limits _{v\in\V\left(\G\sslash\varepsilon_{cs}\right)}\tilde{x}_{v}^{2}d_{v}\left(\G\sslash\varepsilon_{cs}\right)}}.
\end{align*}

For any $\G$ we have $R\left(\mathcal{L}\left(\G\right),x\right)\leq2$
\cite{chung1997spectral}, therefore, 
\begin{equation*}
\sum\limits _{\left\{ u,v\right\} \in\E\left(\G\sslash\varepsilon_{cs}\right)}\left(\tilde{x}_{u}-\tilde{x}_{v}\right)^{2}\leq2\sum\limits _{v\in\V\left(\G\sslash\varepsilon_{cs}\right)}\tilde{x}_{v}^{2}d_{v}\left(\G\sslash\varepsilon_{cs}\right)
\end{equation*}
and 
\begin{equation*}
\frac{1+\frac{4\tilde{x}_{n-1}^{2}}{\sum\limits _{\left\{ u,v\right\} \in\E\left(\G\sslash\varepsilon_{cs}\right)}\left(\tilde{x}_{u}-\tilde{x}_{v}\right)^{2}}}{1+\frac{2\tilde{x}_{n-1}^{2}}{\sum\limits _{v\in\V\left(\G\sslash\varepsilon_{cs}\right)}\tilde{x}_{v}^{2}d_{v}\left(\G\sslash\varepsilon_{cs}\right)}}\geq1,
\end{equation*}
and we obtain that $R\left(\mathcal{L}\left(\G\right),p_{\tilde{\mathcal{F}}_{\pi}}\left(\tilde{x}\right)\right)\geq R\left(\mathcal{L}\left(\G\sslash\varepsilon_{cs}\right),\tilde{x}\right)$
for any cycle invariant single edge contraction.\end{proof}

The only graph contraction interlacing result known to the authors
has been presented by Chen et. al. \citep{chen2004interlacing}:

\begin{theorem}[normalized-Laplacian interlacing disjoint neighborhood contraction]\label{normalized-Laplacian interlacing disjoint neighborhood contraction }
Consider a graph $\G$ and two vetices $u,v\in\V\left(\G\right)$
with corresponding partition $\pi\in\Pi_{n-1}\left(\G\right)$ with
only one non-singlet cell $C_{n-1}\left(\pi\right)=\left\{ u,v\right\} $.
Then if $u,v$ have disjoint neighborhoods, i.e., $\N_{u}\left(\G\right)\cap\left\{ \N_{v}\left(\G\right)\cup v\right\} =\varnothing$, the atom contraction is normalized-Laplacian interlacing, i.e., $\gcont\propto_{\mathcal{L}}\G$.\end{theorem}

\begin{proof}The proof is given in \citep{chen2004interlacing} and
is based on a sequence of min-max ineuqalities and the Courant--Fischer
theorem (Theorem \ref{Courant-Fischer}). In the perspective of this
work, Theorem \ref{normalized-Laplacian interlacing disjoint neighborhood contraction }
can be proven based on Theorem \ref{interlacing graph reduction theorem}
as follows: Let $\gcont$ be an atom contraction with $C_{n-1}\left(\pi\right)=\left\{ u,v\right\} $
and $\N_{u}\left(\G\right)\cap\left\{ \N_{v}\left(\G\right)\cup v\right\} =\varnothing$. We notice that if $\N_{u}\left(\G\right)\cap\left\{ \N_{v}\left(\G\right)\cup v\right\} =\varnothing$ then $\gcont$ is edge-matching (Definition \ref{edge-matching contraction}).
Similar to Proposition \ref{Rayleigh quotient cycle invariant contraction inequalities}
it can then be shown that $R\left(\mathcal{L}\left(\G\right),p_{\mathcal{F}_{\pi}}\left(x\right)\right)\leq R\left(\mathcal{L}\left(\gcont\right),x\right)$
and $R\left(\mathcal{L}\left(\G\right),p_{\tilde{\mathcal{F}}_{\pi}}\left(\tilde{x}\right)\right)\geq R\left(L\left(\gcont\right),x\right)$,
therefore, from Theorem \ref{interlacing graph reduction theorem}
with $\mathcal{A}\equiv\mathcal{F}_{\pi}$ and $\mathcal{B}\equiv\tilde{\mathcal{F}}_{\pi}$
we then get that $\gcont\propto_{\mathcal{L}}\G$.
\end{proof}
 
In this study, based on Theorem \ref{interlacing graph reduction theorem},
we derive two interlacing theorems, Laplacian interlacing for node-removal
equivalent edge-matching contractions (Theorem \ref{node removal contraction laplacian interlacing})
and normalized Laplacian interlacing for cycle-invariant contractions
(Theorem \ref{normalized-Laplacian interlacing cycle-invariant contraction}).

\begin{theorem}[Laplacian interlacing node-removal equivalent contraction]\label{node removal contraction laplacian interlacing}
Consider a graph $\G$ and an edge contraction set $\E_{cs}\in\Xi_{n-r}\left(\G\right)$
for $r<n$. If $\G\sslash\E_{cs}$ is edge-matching (Definition \ref{edge-matching contraction})
and node-removal equivalent (Definition \ref{node-removal equivalent contraction})
then $\G\sslash\E_{cs}\propto_{L}\G$.\end{theorem}

\begin{proof}The contraction $\G\sslash\E_{cs}$ is edge-matching
and node-removal equivalent such that from Proposition \ref{Rayleigh quotient edge-matching and node-removal equivalent contraction  inequalities}
we have $R\left(L\left(\G\right),p_{\mathcal{F}_{\pi}}\left(x\right)\right)\leq R\left(L\left(\G\sslash\E_{cs}\right),x\right)$
and $R\left(L\left(\G\right),p_{\mathcal{F}_{\V_{S}}}\left(x\right)\right)\geq R\left(L\left(\G\sslash\E_{cs}\right),x\right)$.
Therefore, from Theorem \ref{interlacing graph reduction theorem}
with $\mathcal{A}\equiv\mathcal{F}_{\pi}$ and $\mathcal{B}\equiv\mathcal{F}_{\V}$
we then get that $\G\sslash\E_{cs}\propto_{L}\G$.\end{proof}

\begin{theorem}[normalized-Laplacian interlacing cycle-invariant contraction]\label{normalized-Laplacian interlacing cycle-invariant contraction}
Consider a graph $\G$ and an edge contraction set $\E_{cs}\in\Xi_{n-r}\left(\G\right)$
for $r<n$. Then if $\G\sslash\E_{cs}$ is cycle-invariant (Definition
\ref{node-removal equivalent contraction}), $\G\sslash\E_{cs}\propto_{\mathcal{L}}\G$.\end{theorem}

\begin{proof}\sloppy The graph contraction can be performed by a sequence
of atom-contractions (Corollary \ref{atom-contraction sequence}),
therefore, it is sufficient to show that the interlacing property
holds for a single edge-contraction, i.e., $\G\sslash\epsilon_{cs}\propto_{\mathcal{L}}\G$
where $\epsilon_{cs}$ is a single contracted edge. The interlacing
of the sequence will then follow from Proposition \ref{interlacing sequence}.
Let $\G\sslash\varepsilon_{cs}$ be a cycle-invariant edge contraction
with corresponding edge contraction partition $\pi\in\Pi_{n-1}\left(\G\right)$.
Without loss of generality we can label the vertices such that the
contracted edge is $\varepsilon_{cs}=\left\{ x_{n-1},x_{n}\right\} $,
and the anti-partition space is $\tilde{\mathcal{F}}_{\pi}\left(\tilde{x}\right)=\left\{ x\in\mathbb{R}^{n}\,|x_{n-1}=-x_{n}\right\} $
(Eq. \eqref{eq:the partition null-space lifting}). From Proposition
\ref{Rayleigh quotient cycle invariant contraction inequalities}
we have $R\left(\mathcal{L}\left(\G\right),p_{\mathcal{F}_{\pi}}\left(x\right)\right)\leq R\left(\mathcal{L}\left(\G\sslash\epsilon_{cs}\right),x\right)$
and $R\left(\mathcal{L}\left(\G\right),p_{\tilde{\mathcal{F}}_{\pi}}\left(\tilde{x}\right)\right)\geq R\left(L\left(\G\sslash\varepsilon_{cs}\right),x\right)$,
therefore, from Theorem \ref{interlacing graph reduction theorem}
with $\mathcal{A}\equiv\mathcal{F}_{\pi}$ and $\mathcal{B}\equiv\tilde{\mathcal{F}}_{\pi}$
we then get that $\G\sslash\epsilon_{cs}\propto_{\mathcal{L}}\G$.
By performing the contraction sequence (Proposition \ref{interlacing sequence})
we get $\G\sslash\E_{cs}\propto_{\mathcal{L}}\G$.\end{proof}

\begin{coro}Consider a tree $\T=\left(\V,\E\right)$ of order $n$,
and its contraction $\T\sslash\E_{cs}$ for any $\E_{cs}\in\Xi_{n-r}\left(\T\right)$.
Then $\T\sslash\E_{cs}\propto_{\mathcal{L}}\T$.\end{coro}

\begin{proof}The contraction $\T\sslash\E_{cs}$ is cycle-invariant
for any $\E_{cs}\in\Xi_{n-r}\left(\T\right)$ , therefore, from Theorem
\ref{normalized-Laplacian interlacing cycle-invariant contraction}
we obtain that $\T\sslash\E_{cs}\propto_{\mathcal{L}}\T$.\end{proof}

Theorem \ref{node removal contraction laplacian interlacing} and
Theorem \ref{normalized-Laplacian interlacing cycle-invariant contraction}
allow us to try and solve the interlacing graph contraction problem
(Problem \ref{interlacing graph contraction}) for normalized Laplacian
and Laplacian interlacing by finding a cycle-invariant contraction
(Problem \ref{cycle-invariant contraction problem}) or a node-removal
equivalent and edge matching contraction (Problem \ref{node-removal equivalent contraction problem})
respectively.

\begin{prob}[cycle-invariant contraction]\label{cycle-invariant contraction problem}
For a graph $\G$ and a given reduction order $r<n$, find $\E_{cs}\in\Xi_{n-r}\left(\G\right)$
such that $\G\sslash\E_{cs}$ is cycle-invariant (Definition \ref{cycle-invariat contraction}).\end{prob}

\begin{prob}[node-removal equivalent contraction]\label{node-removal equivalent contraction problem}
For a graph $\G$ and a given reduction order $r<n$, find $\E_{cs}\in\Xi_{n-r}\left(\G\right)$
such that $\G\sslash\E_{cs}$ is node-removal equivalent (Definition
\ref{node-removal equivalent contraction}) and edge-matching (Definition
\ref{edge-matching contraction}).\end{prob}

From Proposition \ref{find cycle invariant contraction}, we can obtain
a cycle-invariant contraction, if exists, from the zero rows of the
Tucker representation. A Tucker representation $T_{\left(\T,\C\right)}$
can be calculated by finding a spanning tree $\T\in\mathbb{T}\left(\G\right)$
and then finding the path in $\T$ between the end-nodes of each edge
of $\C\left(\T\right)$ as described in Algorithm \ref{alg:Cycle-invariant-contraction-algo}.
Each path finding operation, e.g., with a depth-first search, is of
complexity $\mathcal{O}\left(n\right)$, and since $\mathcal{O}\left(\left|\E\left(\C\right)\right|\right)=\mathcal{O}\left(\left|\E\left(\G\right)\right|\right)$
the overall complexity of constructing $T_{\left(\T,\C\right)}$ is
$\mathcal{O}\left(mn\right)$, where $m=\left|\E\left(\G\right)\right|$.
Therefore, the cycle-invariant contraction algorithm (Algorithm \ref{alg:Cycle-invariant-contraction-algo})
is of complexity $\mathcal{O}\left(mn\right)$ .

From Proposition \ref{edge-matching node-removal}, we can obtain
a node-removal equivalent and edge matching contraction, if exists,
by first finding for all vertices of $\G$ the connected components
partition $\pi_{cc}\left(\G\backslash v\right)$ and then constructing
$\E_{cs}$ by choosing from all partitions $\left\{ \pi_{cc}\left(\G\backslash v\right)\right\} _{v=1}^{n}$
a subset of cells with a total number of $n-r$ unique nodes (Algorithm
\ref{node-removal equivalent contraction}). Each connected component
finding operation, e.g., with a depth-first search, is of complexity
$\mathcal{O}\left(n+m\right)$, and repeated $n$ times, the overall
complexity of the algorithm is $\mathcal{O}\left(n^{2}+nm\right)$.

The feasibility of the cycle-invariant and node-removal equivalent
problems requires further study. 
\begin{center}
\begin{algorithm}[h]
\caption{Cycle-invariant contraction algorithm\label{alg:Cycle-invariant-contraction-algo}}

\textbf{Input: }graph\textbf{ $\G$ }of order\textbf{ $n$, }required
reduction order\textbf{ $r$} 
\begin{enumerate}
\item \textbf{Find} a spanning tree $\T\in\mathbb{T}\left(\G\right)$ and
the co-tree $\C\left(\T\right)$. 
\item \textbf{Calculate} the tucker representation \textbf{$T_{\left(\T,\mathcal{C}\right)}$
}(Definition \ref{Tucker_representation}). 
\item \textbf{Choose $n-r$ }cycle-invariant edges from the zero rows of
$T_{\left(\T,\mathcal{C}\right)}$ and obtain $\E_{cs}$. 
\end{enumerate}
\textbf{Output: $\G_{r}=\G\sslash\E_{cs}$} 
\end{algorithm}
\par\end{center}

\begin{center}
\begin{algorithm}[h]
\caption{Node-removal equivalent contraction algorithm\label{alg:Node-removal-equivalent-contract}}

\textbf{Input: }graph\textbf{ $\G$ }of order\textbf{ $n$, }required
reduction order\textbf{ $r$} 
\begin{enumerate}
\item \textbf{For }$v\in\V\left(\G\right)$\textbf{: Calculate} $\pi_{cc}\left(\G\backslash v\right)$,
the connected components partition of $\G\backslash v$. 
\item \textbf{Choose }a subset of cells $\mathcal{S}\subseteq\left\{ \pi_{cc}\left(\G\backslash v\right)\right\} _{v=1}^{n}$
with a total number of $n-r$ unique nodes\textbf{.} 
\item \textbf{Construct} $\E_{cs}=\cup_{C_{v}\in\mathcal{S}}\E\left(\G\left[C_{v}\cup v\right]\right)$. 
\end{enumerate}
\textbf{Output: $\G_{r}=\G\sslash\E_{cs}$} 
\end{algorithm}
\par\end{center}

\section{Case Studies\label{sec:Case-Studies}}

As a small-scale normalized Laplacian interlacing example, we consider
a graph of order $6$ presented in Figure \ref{fig:Small-scale-normalized_Laplacian},
and we require the reduced graph to be of order $r=4$. A cycle-invariant
graph contraction is then performed with two edges (Figure \ref{fig:Small-scale-normalized_Laplacian}).
The resulting reduced graph (Figure \ref{fig:Small-scale-normalized_Laplacian})
has normalized-Laplacian spectra $\left\{ \lambda_{k}\left(\mathcal{L}\left(\G_{r}\right)\right)\right\} _{k=1}^{r}$
given in Figure \ref{fig:Reduced-order normalized-Laplacian spectra}
with the upper and lower interlacing bounds $\lambda_{k}\left(\mathcal{L}\left(\G\right)\right)$
and $\lambda_{n-r+k}\left(\mathcal{L}\left(\G\right)\right)$. Since
$\G\sslash\E_{cs}$ is cycle-invariant, then as according to Theorem
\ref{normalized-Laplacian interlacing cycle-invariant contraction},
we get $\G\sslash\E_{cs}\propto_{\mathcal{L}}\G$ and the reduced-order
spectra is within the interlacing bounds (Figure \ref{fig:Reduced-order normalized-Laplacian spectra}). 
\begin{center}
\begin{figure}
\centering{}\includegraphics[width=4in]{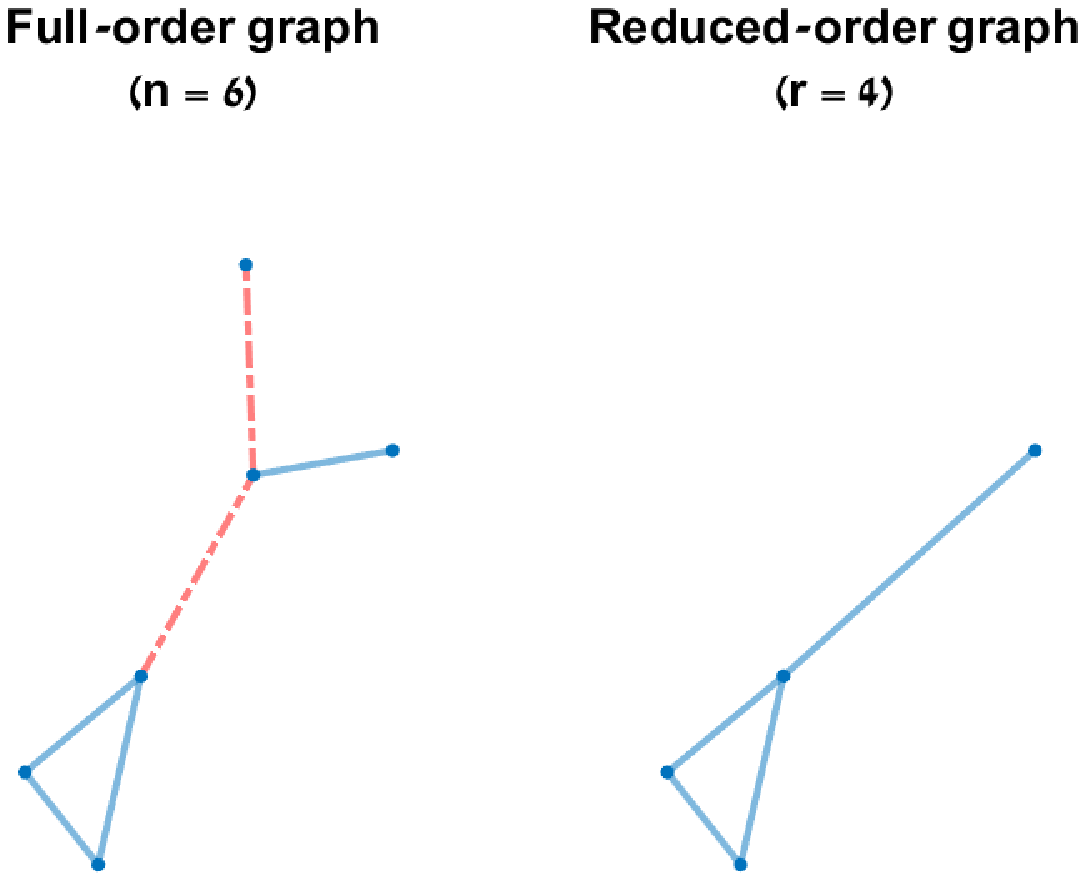}\caption{Small scale normalized-Laplacian interlacing graph contraction (contracted
edges dashed-red).\label{fig:Small-scale-normalized_Laplacian}}
\end{figure}
\par\end{center}

\begin{center}
\begin{figure}
\centering{}\includegraphics[width=4in]{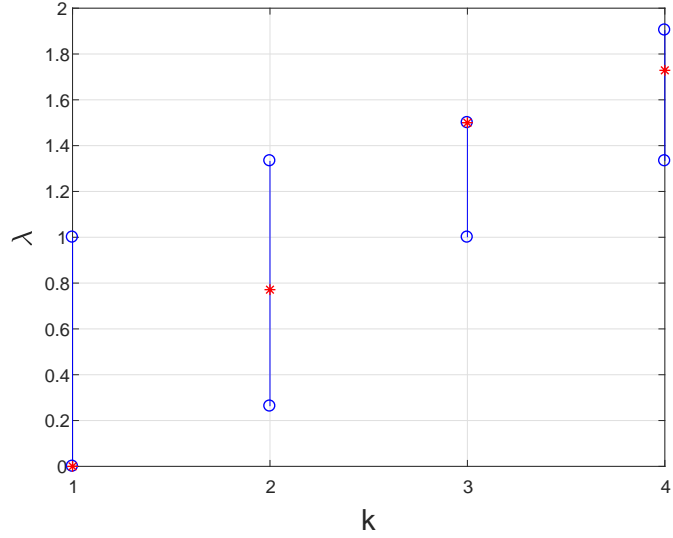}\caption{Reduced-order normalized-Laplacian spectra (stared-red) and interlacing
bounds (circled-blue).\label{fig:Reduced-order normalized-Laplacian spectra}}
\end{figure}
\par\end{center}

As a small-scale Laplacian interlacing example, we consider a graph
of order $6$ presented in Figure \ref{fig:Small-scale-Laplacian}
and require the reduction to be of order $r=4$. For this case the
only node-removal equivalent and edge-matching contraction is with
the three edges shown in Figure \ref{fig:Small-scale-Laplacian}.
The resulting reduced graph (Figure \ref{fig:Small-scale-Laplacian})
has Laplacian spectra given in Figure \ref{fig:Laplacian-interlacing-1-1}
with the interlacing bounds $\lambda_{k}\left(L\left(\G\right)\right)$
and $\lambda_{n-r+k}\left(L\left(\G\right)\right)$. Since $\G\sslash\E_{cs}$
is node-removal equivalent and edge-matching, then as according to
Theorem \ref{node removal contraction laplacian interlacing} we get
$\G\sslash\E_{cs}\propto_{L}\G$ and the reduced-order Laplacian spectra
is within the interlacing bounds (Figure \ref{fig:Laplacian-interlacing-1-1}).
Notice that for this case there is no cycle-invariant contraction,
and for the same choice of $\E_{cs}$ (Figure \ref{fig:Small-scale-Laplacian})
the reduced-order normalized-Laplacian does not interlace with the
full-order normalized-Laplacian as $\lambda_{4}\left(\mathcal{L}\left(\G_{r}\right)\right)>\lambda_{6}\left(\mathcal{L}\left(\G\right)\right)$
(Figure \ref{fig:Reduced-order-normalized-Laplacian spectra_n1}). 
\begin{center}
\begin{figure}
\centering{}\includegraphics[width=4in]{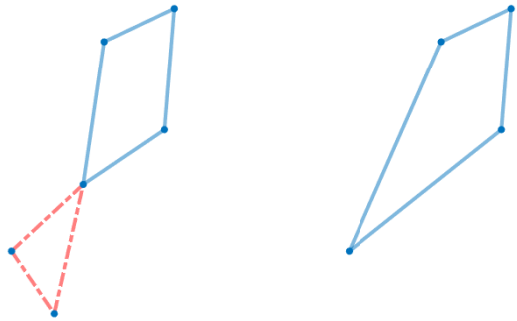}\caption{Small scale Laplacian interlacing graph contraction (contracted edges
dashed-red).\label{fig:Small-scale-Laplacian}}
\end{figure}
\par\end{center}

\begin{center}
\begin{figure}
\centering{}\includegraphics[width=4in]{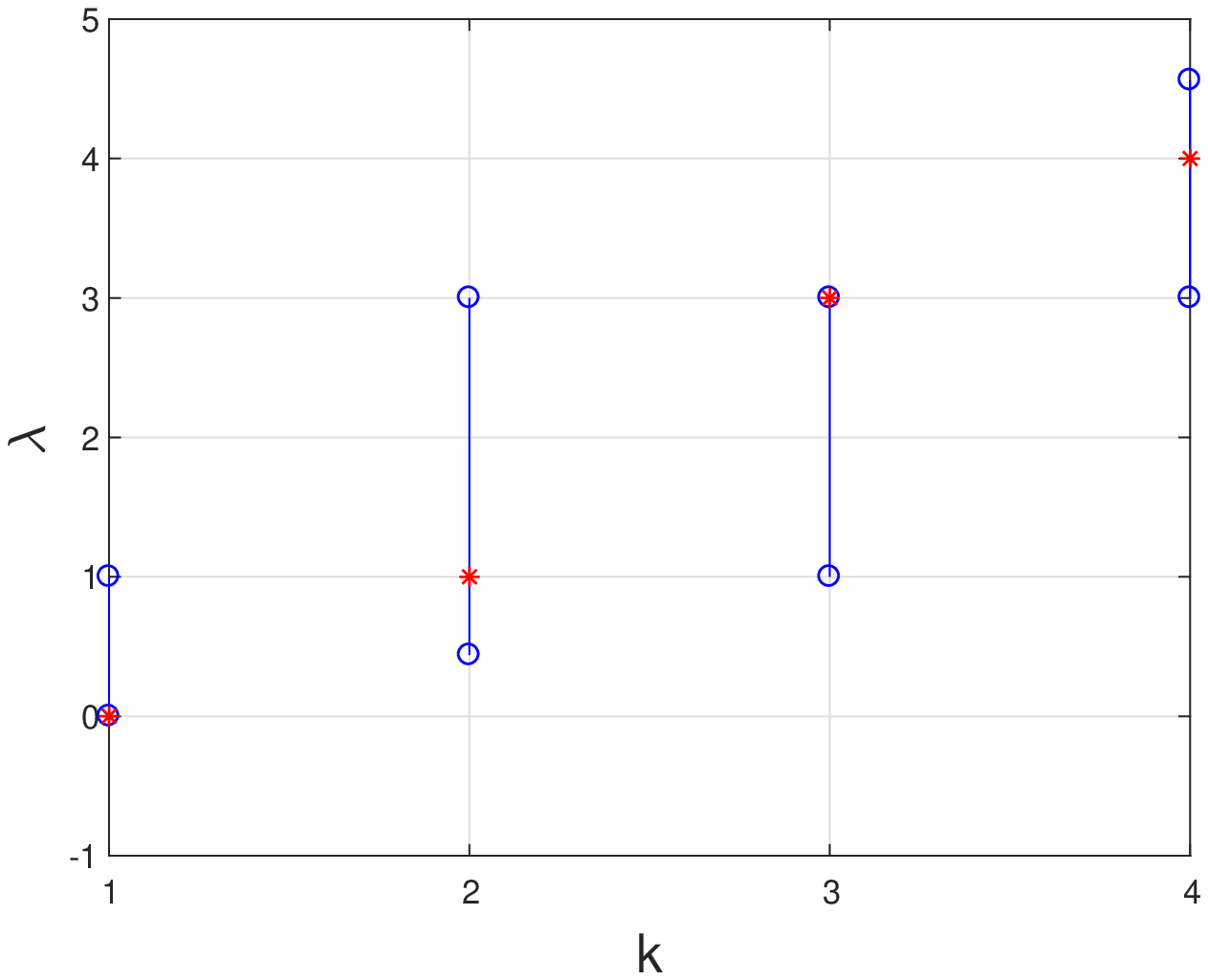}\caption{Reduced-order Laplacian spectra (stared-red) and interlacing bounds
(circled-blue).\label{fig:Laplacian-interlacing-1-1}}
\end{figure}
\par\end{center}

\begin{center}
\begin{figure}
\centering{}\includegraphics[width=4in]{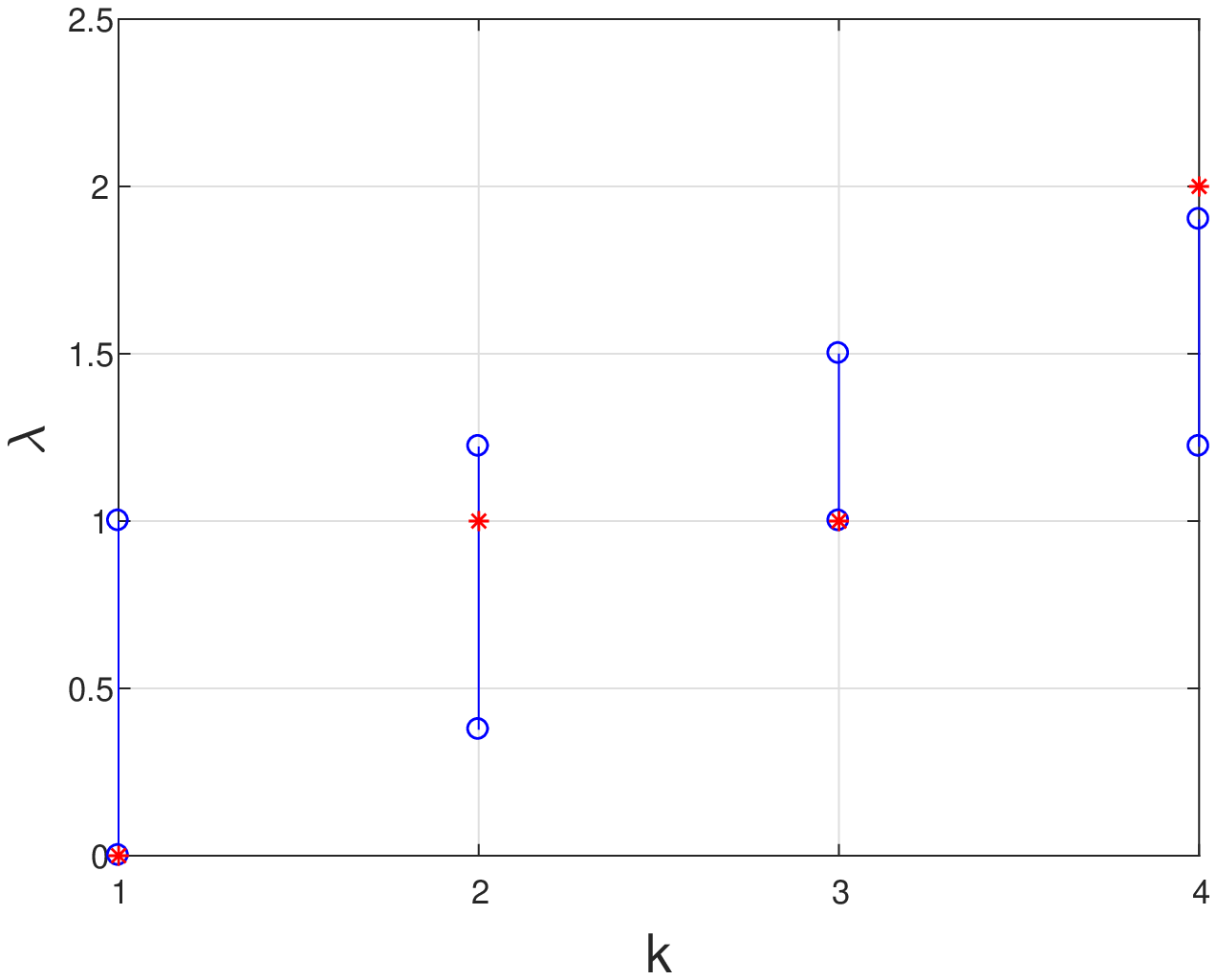}\caption{Reduced-order normalized-Laplacian spectra (stared-red) and interlacing
bounds (circled-blue).\label{fig:Reduced-order-normalized-Laplacian spectra_n1}}
\end{figure}
\par\end{center}

As a larger and more complicated example, a random tree of order $50$
is created and $10$ cycle-completing edges are randomly added to
it resulting in a graph of order $50$ with $59$ edges (Figure \ref{fig:Large-scale-normalized}).
The required reduction order is $r=30$. Using the cycle-invariant
contraction algorithm (Algorithm \ref{alg:Cycle-invariant-contraction-algo})
an edge-contraction set $\E_{cs}$ with $n-r=20$ edges is chosen
from the edges of $\G$ (Figure \ref{fig:Large-scale-normalized}),
and the graph contraction is performed. As according to Theorem \ref{normalized-Laplacian interlacing cycle-invariant contraction},
the resulting reduced-order graph $\G_{r}=\G\sslash\E_{cs}$ is normalized-Laplacian
interlacing with $\G$ and the reduced spectra is within the interlacing
bounds (Figure \ref{fig:Normalized-Laplacian-interlacing}).

Using the node-removal equivalent contraction algorithm (Algorithm
\ref{alg:Node-removal-equivalent-contract}) a different edge-contraction
set $\E_{cs}$ with $n-r=20$ edges is chosen from the edges of $\G$
(Figure \ref{fig:Large-scale-Laplacian}), and the graph contraction
is performed. As according to Theorem \ref{node removal contraction laplacian interlacing},
the resulting reduced order graph $\G_{r}=\G\sslash\E_{cs}$ (Figure
\ref{fig:Large-scale-Laplacian}) is Laplacian interlacing with $\G$
and the reduced spectra is within the interlacing bounds (Figure \ref{fig:Laplacian-interlacing-1}). 
\begin{center}
\begin{figure}
\centering{}\includegraphics[width=4in]{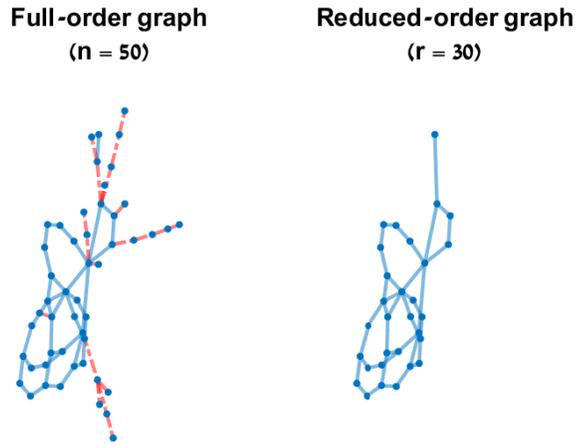}\caption{Large scale normalized-Laplacian interlacing graph contraction (contracted
edges dashed-red).\label{fig:Large-scale-normalized}}
\end{figure}
\par\end{center}

\begin{center}
\begin{figure}
\centering{}\includegraphics[width=4in]{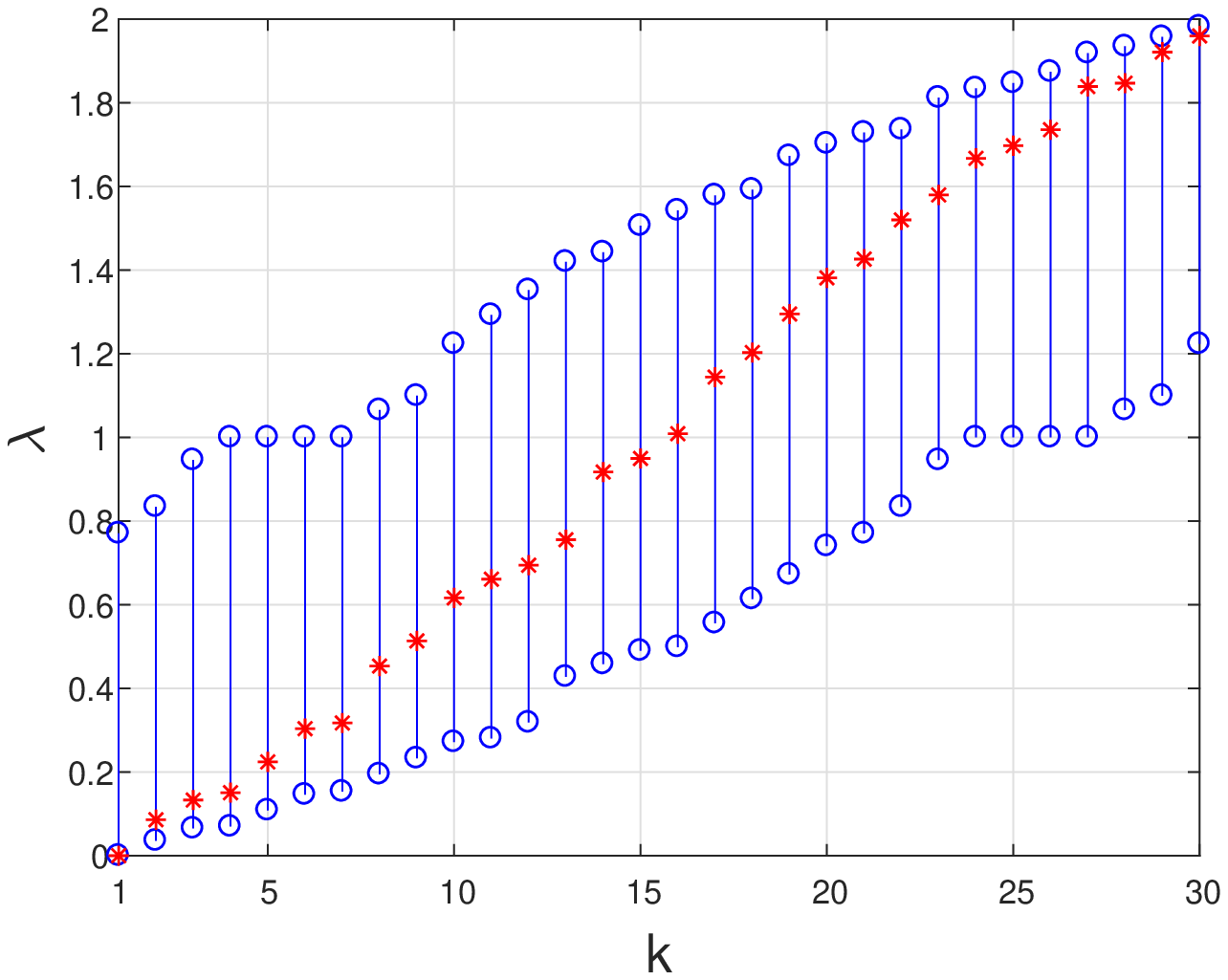}\caption{Reduced-order normalized-Laplacian spectra (stared-red) and interlacing
bounds (circled-blue).\label{fig:Normalized-Laplacian-interlacing}}
\end{figure}
\par\end{center}

\begin{center}
\begin{figure}
\centering{}\includegraphics[width=4in]{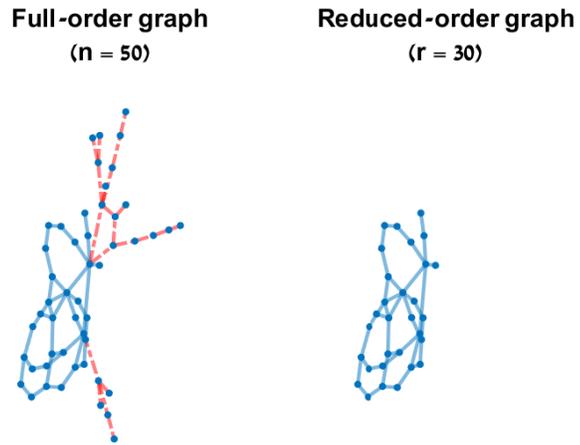}\caption{Large scale Laplacian interlacing graph contraction.\label{fig:Large-scale-Laplacian}}
\end{figure}
\par\end{center}

\begin{center}
\begin{figure}
\centering{}\includegraphics[width=4in]{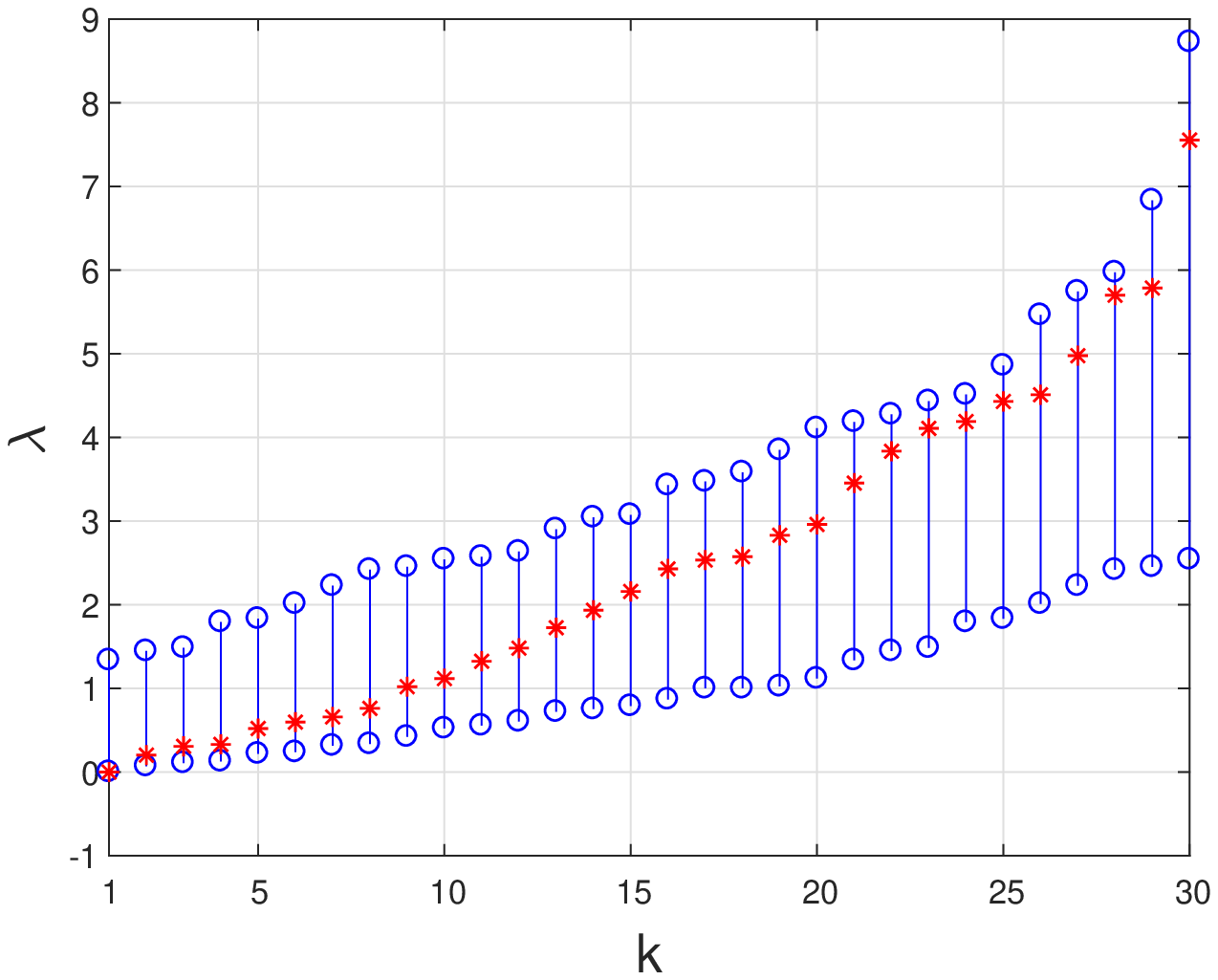}\caption{Reduced-order Laplacian spectra (stared-red) and interlacing bounds
(circled-blue).\label{fig:Laplacian-interlacing-1}}
\end{figure}
\par\end{center}

\bibliographystyle{elsarticle-num}
\bibliography{laa_refs}

\end{document}